\documentclass[12pt]{amsart}
\usepackage{amsmath}
\usepackage{amssymb}
\newtheorem{theorem}{Theorem}[section]
\newtheorem{corollary}[theorem]{Corollary}
\newtheorem{lemma}[theorem]{Lemma}

\theoremstyle{definition}
\newtheorem{definition}[theorem]{Definition}
\theoremstyle{remark}
\newtheorem{remark}[theorem]{\bf{Remark}}
\numberwithin{equation}{section}
\theoremstyle{remark}

\newcommand{\R}{\mathbb R}
\newcommand{\C}{\mathbb C}
\newcommand{\N}{\mathbb N}

\begin{document}

\title[Fatou theorem and its converse] {Fatou theorem and its converse for positive eigenfunctions of the Laplace-Beltrami operator on Harmonic $NA$ groups}
\author[S.K. Ray and J. Sarkar]{Swagato K. Ray and Jayanta Sarkar}
\address[S.K. Ray]{Stat Math Unit, Indian Statistical Institute, 203 B. T. Road, Kolkata 700108}
\email{swagato@isical.ac.in}
\address[J. Sarkar]{Stat Math Unit, Indian Statistical Institute, 203 B. T. Road, Kolkata 700108}
\email{jayantasarkarmath@gmail.com}
\begin{abstract}
We prove a Fatou-type theorem and its converse for certain positive eigenfunctions of the Laplace-Beltrami operator $\mathcal{L}$ on a Harmonic $NA$ group. We show that a positive eigenfunction $u$ of $\mathcal{L}$ with eigenvalue $\beta^2-\rho^2$, where $\beta\in (0,\infty)$, has admissible limit in the sense of Korányi, precisely at those boundary points where the strong derivative of the boundary measure of $u$ exists. Moreover, the admissible limit and the strong derivative are the same. This extends a result of Ramey and Ullrich regarding nontangential convergence of positive harmonic functions on the Euclidean upper half space. 
\end{abstract}
\maketitle
\section{Introduction}
Throughout this article, whenever we say $\mu$ is a measure on some locally compact Hausdroff space, we will always mean $\mu$ is a complex Borel measure or a signed Borel measure such that the total variation $|\mu|$ is locally finite, that is, $|\mu|(K)$ is finite for all compact sets $K$. If $\mu(E)$ is nonnegative for all Borel measurable sets $E$ then $\mu$ will be called a positive measure. Our starting point is a classical result of Fatou which relates the differentiability property of a measure $\mu$ on $\R^l$, $l\in\N$ at a given boundary point $x_0\in\R^l$ with the boundary behavior of its Poisson integral $P\mu$ at $x_0$. In order to state Fatou's theorem, we recall that the Poisson integral $P\mu$ on the Euclidean upper half space $\R^{l+1}_+=\{(x,y)\mid x\in\R^l,y>0\}$ of a measure on $\R^l$ is defined by the convolution $P\mu(x,y)=\mu\ast P_y(x)$, $(x,y)\in\R^{l+1}_+$, where the Poisson kernel $P_y$ is given by
\begin{equation*}
P_y(x)=\pi^{-(l+1)/2}\Gamma\left(\frac{l+1}{2}\right)\frac{y}{(y^2+\|x\|^2)^{\frac{l+1}{2}}},\:\:\:\:\:(x,y)\in\R^{l+1}_+.
\end{equation*}
It is well-known that if $P\mu(x_0,y_0)$ is finite at some point $(x_0,y_0)\in\R^{l+1}_+$, then $P\mu$ is well-defined in $\R^{l+1}_+$ and is a harmonic function there. A well-known result of Fatou says that if for some $x_0\in\R^l$ and $L\in \C$ we have
\begin{equation*}
D_{sym}\mu(x_0):=\lim_{r\to 0}\frac{\mu(B(x_0,r))}{|B(x_0,r)|}=L,
\end{equation*}
where $|E|$ is the Lebesgue measure of a measurable set $E\subset\R^l$, then it follows that
\begin{equation*}
\lim_{y\to 0}P\mu(x_0,y)=L.
\end{equation*}
The quantity $D_{sym}\mu(x_0)$, if it exists, is known as the symmetric derivative of $\mu$ at $x_0$. A generalization of this result was proved by Saeki in \cite{Sa} for more general approximate identities instead of the Poisson kernel (see \cite[Theorem 2.3]{Sar1} for the precise statement). For $l=1$, it was shown by Loomis \cite{L} that the converse of the above result of Fatou is false in general, but it remains true if $\mu$ is assumed to be a positive measure. Later, Rudin \cite[Theorem A]{Ru} extended the result of Loomis by proving converse of Fatou's theorem for positive measures on $\R^l$. It is well-known that positive harmonic functions in $\R^{l+1}_+$ are characterized by positive measures on the boundary $\R^l$ of $\R^{l+1}_+$ in the following sense \cite[Theorem 7.26]{ABR}: 
\emph{if $u$ is a positive harmonic function in $\R^{l+1}_+$, then there exists a unique positive measure $\mu$ on $\R^l$ (known as the boundary measure of $u$) and a nonnegative constant $C$ such that 
\begin{equation*}
u(x,y)=Cy+P\mu(x,y),\:\:\:\text{for all}\:\:\:(x,y)\in \R^{l+1}_+.
\end{equation*}}
Using this characterization and Loomis-Rudin theorem, we obtain the following equivalence of boundary behavior of a positive harmonic function in $\R^{l+1}_+$ along the normal at a given boundary point with the existence of symmetric derivative of its boundary measure at that point.
\begin{theorem}\label{fatourudin}
Let $u$ be a positive harmonic function in $\R^{l+1}_+$ with the boundary measure $\mu$. Suppose that $x_0\in\R^l$, $L\in [0,\infty)$. Then $D_{sym}\mu(x_0)=L$, if and only if 
\begin{equation*}
\lim_{y\to 0}u(x_0,y)=L.
\end{equation*}
\end{theorem}
Recently, one of the authors \cite[Theorem 2.10]{Sar1} extended the Loomis-Rudin theorem for more general approximate identities and then used this result to prove an analogue of Theorem \ref{fatourudin} for positive eigenfunctions of the Laplace-Beltrami operator $\Delta_{\mathbb{H}^l}$ on real hyperbolic spaces $\mathbb{H}^l=\{(x,y)\mid x\in \R^{l-1}, y\in (0,\infty)\}$, $l\geq 2$, equipped with the standard hyperbolic metric $ds^2=y^{-2}(dx^2+dy^2)$. 
However, a counter example by Rudin shows that the analogue of Theorem \ref{fatourudin} is false for complex hyperbolic spaces \cite[Example 5.4.13]{Ru2}. 
In this paper we will talk about a result by Ramey and Ullrich \cite{UR} dealing with positive harmonic functions on $\R^{l+1}$, which is closely related to Theorem \ref{fatourudin}.
\begin{theorem}[{{\cite[Theorem 2.2]{UR}}}]\label{rameyullrich}
Let $u$ be a positive harmonic function in $\R^{l+1}_+$ with the boundary measure $\mu$. Suppose that $x_0\in\R^l$, $L\in [0,\infty)$. Then $\mu$ has strong derivative $L$ at $x_0$, if and only if $u$ has nontangential limit $L$ at $x_0$.
\end{theorem} 
According to \cite{UR}, a measure $\mu$ is said to have the strong derivative $L$ at $x_0$, if
\begin{equation*}
\lim_{r\to 0}\frac{\mu(x_0+rB)}{|rB|}=L
\end{equation*}
holds for every open ball $B\subset\R^l$, where $rB=rB(x,s)=B(rx, rs)$, $x\in \R^l$, $s\in (0,\infty)$. A generalization of Theorem \ref{rameyullrich} has recently been proved by Logunov \cite[Theorem 10]{Lo} for positive solutions of more general uniformly elliptic operators in $\R^{l+1}_+$.  

In this article, we will prove an analogue of Theorem \ref{rameyullrich} for certain positive eigenfunctions $u$ of the Laplace-Beltrami operator on a class of Riemannian manifolds which includes all Riemannian symmetric spaces of noncompact type with real rank one (excluding real hyperbolic spaces), namely, the harmonic $NA$ groups (also known as Damek-Ricci spaces). Here $N$ is a group of Heisenberg type, $A=(0,\infty)$ acts on $N$ as nonisotropic dilation and $S=NA$ is the semidirect product under the action of dilation. For unexplained notions and terminologies we refer the reader to Section $2$. It is known that $S$ is a Riemannian manifold with respect to a metric which is left invariant under the action of $S$ (see \cite{ADY}). In this context, the correct analogue of the notion of nontangential convergence is the notion of admissible convergence introduced by Korányi \cite{Kocom} for complex hyperbolic spaces. In \cite{Kosym}, Korányi extended the notion of admissible convergence from complex hyperbolic spaces to rank one Riemannian symmetric space of noncompact type and proved the Fatou theorem regarding admissible convergence of the Poisson integral of an integrable function almost everywhere on the Furstenberg boundary. For the definition of admissible convergence we refer the reader to Definition \ref{impdefnna}, i). There is an extensive literature concerning extension and generalization of this result (see for example \cite{KoTa}, \cite{Mi}, \cite{Sch}, \cite{Sjo} and references therein). This version of Fatou theorem regarding almost everywhere admissible convergence was further extented by Michelson \cite{Mi} for eigenfunctions of the Laplace-Beltrami operator on Riemannian symmetric spaces of noncompact type. However, this body of literature does not seem to contain any result relating the notion of admissible convergence of positive eigenfunctions of the Laplace-Beltrami operator and differentiation property of boundary measures on Riemannian symmetric spaces of noncompact type or more generally on $S$ at a given point in the relevant boundary, which is $N$. Let us now briefly describe the main result of this article. Let $u$ be a positive eigenfunction of the Laplace-Beltrami operator $\mathcal{L}$ on $S$, a Harmonic $NA$ group. By a result in \cite{DR} (see Lemma \ref{positiveeigen}), $u$ is essentially given by a Poisson type integral of a positive measure $\mu$ defined on $N$. On the other hand, since $N$ is a group of Heisenberg type, it comes equip with a homogeneous norm giving rise to a quasi-metric invariant under the action of $N$. These structural facts help us to define an analogue of the notion of strong derivative on the group $N$. With the aid of these basic notions, the main result (Theorem \ref{mainthna}) of this paper says that given a point $n_0\in N$, a suitable positive eigenfunction $u$ of the Laplace-Beltrami operator has admissible limit $L$ at $n_0$ if and only if the strong derivative of $\mu$ at $n_0$ is equal to $L$. Our proof is modelled on the original proof by Ramey and Ullrich but application of some recent results  are also necessary. For example,  a recent result of Bar \cite[Theorem 4]{B} on generalization of Montel’s theorem plays an important role in the proof of the main theorem.

Throughout the paper, we reserve the letters $c$, $C$, $C'$ for positive constants whose values are unimportant and can change at each occurrence, unless otherwise stated. We also use notation like $C_{\delta}$ to indicate the dependency on the parameters $\delta$. 

This paper is organised as follows. In section 2, we will discuss some basic information about Harmonic $NA$ groups, generalized Poisson kernel and Poisson integral on these groups. In section 3 we prove some results which are crucial for the proof of the main theorem. The statement and proof of the main theorem (Theorem \ref{mainthna}) is given in the last section. 

\section{Preliminaries of Harmonic $NA$ groups}
A harmonic $NA$ group is a solvable Lie group as well as a Harmonic manifold. Their distinguished prototypes are the Riemannian symmetric spaces of noncompact type with real rank one. However, the rank one Riemannian symmetric spaces of noncompact type form a very small subclass in the class of Harmonic $NA$ groups \cite[1.10]{ADY}. In the following, we discuss them in detail. Most of these material can be found in \cite{FS, ADY, ACD}.

Let $\mathfrak n$ be a two-step real nilpotent Lie algebra equipped with an inner product $\langle,\rangle$. Let $\mathfrak z$ be the centre of $\mathfrak n$ and $\mathfrak v$ its orthogonal complement. We say that $\mathfrak n$ is an $H$-type algebra if for every $Z\in \mathfrak z$  the map $J_Z:\mathfrak v\to \mathfrak v$ defined by
\begin{equation*}
\langle J_Z X, Y\rangle=\langle [X, Y], Z\rangle,\ X, Y\in \mathfrak v,
\end{equation*}
satisfies the condition
$J_Z^2=-|Z|^2I_{\mathfrak v}$, $I_{\mathfrak v}$ being the identity operator on $\mathfrak v$. A connected and simply connected Lie group $N$ is called an $H$-type group if its Lie algebra is a $H$-type algebra. Since $\mathfrak{n}$ is nilpotent, the exponential map is a diffeomorphism and hence we can parametrize elements of $N=\exp\mathfrak{n}$ by $(X,Z)$ for $X\in\mathfrak{v},\:Z\in\mathfrak{z}$.  It follows from the Baker-Campbell-Hausdorff formula that the group law of $N$ is given by
$$(X,Z)(X',Z')=(X+X',Z+Z'+\frac{1}{2}[X,X']).$$
When $\mathfrak z=\R$, $\mathfrak v=\R^{2l}$ and for $s\in\R$, $J_s:\R^{2l}\to \R^{2l}$ is given by
\begin{equation*}
J_s(x,y)=(-sy,sx),\:\:\:\:\:x\in\R^l,\:y\in\R^{l},
\end{equation*} 
then we get the Heisenberg group $H^l$ which is the prototype of a $H$-type group. $H$-type groups forms an important class of stratified groups (also known as Carnot groups) \cite[Remark 18.1.7]{BLU}. 
We also note that the Lebesgue measure $dXdZ$ is a Haar measure on $N$ and we denote it by $m$.
The group $A=(0,\infty)$ acts on an $H$-type group $N$ by nonisotropic dilation:
\begin{equation}\label{nonisotropic}
\delta_a(n)=\delta_a(X,Z)=(\sqrt{a}X,aZ),\:\:a\in A,\:n=(X,Z)\in N.
\end{equation}  
A Harmonic $NA$ group $S$ is the semidirect product of a $H$-type group $N$ and $A$ under the above action. Thus, the multiplication on $S$ is given by
$$(X,Z,a)(X',Z',a')=(X+\sqrt{a}X',Z+aZ'+\frac{1}{2}\sqrt{a}[X,X'],aa').$$ 
Then $S$ is a solvable, connected and simply connected Lie group having Lie algebra $\mathfrak{s}=\mathfrak{n}\oplus\mathfrak{z}\oplus\R$ with Lie bracket 
$$[(X,Z,u),(X',Z',u')]=(\frac{1}{2}uX'-\frac{1}{2}u'X,uZ'-u'Z+[X,X'],0).$$
We shall write $(n,a)=(X,Z,a)$ for the element $\left(\exp(X+Z),a\right),\:a\in A,\:X\in\mathfrak{v},\:Z\in\mathfrak{z}.$ We note that for any $Z\in\mathfrak{z}$ with $\|Z\|=1$, $J_Z^2=-I_{\mathfrak{v}}$ and hence $\mathfrak{v}$ is even dimensional. We suppose that $dim\:\mathfrak{v}=2p,\:dim\:\mathfrak{z}=k$. Then $Q=p+k$ is called the homogeneous dimension of $N$. The importance of homogeneous dimension stems from the following relation 
\begin{equation}\label{measuredilation}
m\left(\delta_a(E)\right)=a^Qm(E),
\end{equation}
which holds for all measurable sets $E\subseteq N$ and $a\in A$. For convenience, we shall also use the notation $\rho=Q/2$. We denote by $e$ the identity element $(\underline{0},1)$ of $S$, where $\underline{0}$, $1$ are the identity elements of $N$ and $A$ respectively. We note that $\rho$ corresponds to the half-sum of positive roots when $S=G/K$, is a rank one symmetric space of noncompact type. The group $S$ is equipped with the left-invariant Riemannian metric induced by
$$\langle (X,Z,u),(X',Z',u')\rangle_S=\langle X,X'\rangle+\langle Z,Z'\rangle+uu'$$ 
on $\mathfrak{s}$. The associated left-invariant Haar measure $dx$ on $S$ is given by
$$dx=a^{-Q-1}dXdZda,$$
where $dX,\:dZ,\:da$ are the Lebesgue measures on $\mathfrak{v},\:\mathfrak{z},\:A$ respectively. Being a stratified Lie group, $N$ always admit homogeneous norms with respect to the family of dilations $\{\delta_a\mid a\in A\}$ \cite[P.8-10]{FS}. We recall that a continuous function $d:N\to[0,\infty)$ is said to be a homogeneous norm on $N$ with respect to the family of dilations $\{\delta_a\mid a\in A\}$ if $d$ satisfies the following \cite[P.8]{FS}:
\begin{enumerate}
	\item[i)]$d$ is smooth on $N\setminus\{0\}$;
	\item[ii)] $d(\delta_a(n))=rd(n)$,  for all  $a>0,\:n\in N$;
	\item [iii)]$d(n^{-1})=d(n)$, for all  $n\in N$;
	\item[iv)]$d(n)=0$  if and only if  $n=\underline{0}$.
\end{enumerate} 
It is known that \cite[Proposition 1.6]{FS} for any homogeneous norm $d$ on $N$, there exists a positive constant $\tau_d\geq1$, such that 
\begin{equation}\label{quasinorm}
d(nn_1)\leq \tau_d\left[d(n)+d(n_1)\right],\;\:\:\:n\in N, n_1\in N.
\end{equation}
A homogeneous norm $d$ defines a left invariant quasi-metric on $N$, denoted by $\bf{d}$, as follows:
\begin{equation*}
{\bf d}(n_1,n_2)=d(n_1^{-1}n_2),\:\:\:\: n_1\in N,\:n_2\in N.
\end{equation*}
In fact, one can easily verify the following from the definition of $d$ and from (\ref{quasinorm}) that
\begin{enumerate}
\item [i)] ${\bf d}(n_1,n_2)={\bf d}(n_2,n_1),\:\:\:\text{for all}\:\:n_1,\:n_2\in N$.
\item[ii)] ${\bf d}(nn_1,nn_2)={\bf d}(n_1,n_2),\:\:\:\text{for all}\:\:n_1,\:n_2,\:n\in N.$	
\item [iii)] For all $n_1,\:n_2,\:n\in N$
\begin{equation}\label{quasitriangle}
{\bf d}(n_1,n_2)\leq \tau_d\left[{\bf d}(n_1,n)+{\bf d}(n,n_2)\right].
\end{equation}
\end{enumerate}
From \cite[P.230]{BLU} it is known that any two homogeneous norms $d_1$ and $d_2$ on $N$ are equivalent in the sense that there exists a positive constant $B$ such that 
\begin{equation*}
B^{-1}d_1(n)\leq d_2(n)\leq Bd_2(n),\:\:\:\text{for all}\:\:n\in N.
\end{equation*}
We will work with the following homogeneous norm \cite[P.1918]{DK}:
\begin{equation}\label{cannonicalnorm}
d(n)=d(X,Z)=(\|X\|^4+16\|Z\|^2)^{\frac{1}{2}},\:\:\:\:\:n=(X,Z)\in N,
\end{equation} 
where $\|X\|$, $\|Z\|$ are usual Euclidean norms of $X\in\mathfrak{v}\cong\R^{2p}$ and $Z\in\mathfrak{z}\cong\R^k$ respectively. For the sake of simplicity, we write $\tau=\tau_d$. For $n\in N$ and $r>0$,  the $d$-ball centered at $n$ with radius $r$ is defined as
\begin{equation*}
B(n,r)=\{n_1\in N\mid{\bf d}(n,n_1)<r\}=\{n_1\in N\mid d(n_1^{-1}n)<r\}.
\end{equation*}
Then $\overline{B}(n,r)=\{n_1\in N\mid {\bf d}(n,n_1)\leq r \}$, which is a compact subset of $N$ \cite[Lemma 1.4]{FS}.
It follows that $B(n,r)$ is the left translate by $n$ of the ball $B(\underline{0},r)$ which in turn, is the image under $\delta_r$ of the ball $B(\underline{0},1)$. This shows, using (\ref{measuredilation}), that 
\begin{equation*}
m\left(B(n,r)\right)=m\left(B(\underline{0},r)\right)=m\left(B(\underline{0},1)\right)r^Q,\:\:\:\text{for all $n\in N$, $r>0$}.
\end{equation*}
We also observe that if $B=B(n,t)$ for some $n\in N$, $t>0$, then 
\begin{equation*}
\delta_r(B)=B(\delta_r(n),rt),\:\:\:\:\:\text{for all $r>0$}. 
\end{equation*}
We recall the following formula for integration (an analogue of polar coordinate) which can be used in order to determine the integrability of functions on $N$ (\cite[Proposition 1.15]{FS}): for all $g\in L^1(N)$,
\begin{equation}\label{polarcordinatena}
\int_Ng(n)\:dm(n)n=\int_{0}^{\infty}\int_{\Omega}g(\delta_r(\omega))r^{Q-1}\:d\sigma(\omega)\:dr,
\end{equation}
where $\Omega=\{\omega\in N\mid d(\omega)=1\}$ and $\sigma$ is a unique positive Radon measure on $\Omega$. For a function $\psi$ defined on $N$, we define for $a>0$,
\begin{equation}\label{dilationoffunction}
\psi_a(n)=a^{-Q}\psi\left(\delta_{{\frac{1}{a}}}(n)\right),\:\:\:\:\:n\in N.
\end{equation}
If $g$ is a measurable function on $N$ and $\mu$ is a measure on $N$, their convolution $\mu\ast g$ is defined by 
\begin{equation*}
\mu\ast g(n)=\int_{N}g(n_1^{-1}n)\:d\mu(n_1),
\end{equation*}
provided the integrals converges. When $d\mu=f\:dm$, we simply denote the above convolution by $f\ast g$. We refer to \cite[P.15-18]{FS} for basic properties of convolution on $N$.
\begin{remark}\label{approximateidentity}
\begin{enumerate}
\item [i)] It follows from (\ref{measuredilation}) that if $\psi\in L^1(N)$, then for all $a\in A$
\begin{equation*}
\int_{N}\psi_a(n)\:dm(n)=\int_{N}\psi(n)\:dm(n).  
\end{equation*} 
\item [ii)] Suppose that $\psi\in L^1(N)$ with $\int_{N}\psi(n)\:dm(n)=1$. Then $\{\psi_a\}_{a>0}$ is an approximate identity on $N$ \cite[Proposition 1.20]{FS}. In particular, for $f\in C_c(N)$, it follows that $f\ast\psi_a\to f$, as $a\to 0$, uniformly on $N$.
\end{enumerate}
\end{remark}
We now describe the Laplace-Beltrami operator on $S$. Let $\{e_i\mid 1\leq i\leq 2p\}$, $\{e_r\mid 2p+1\leq r\leq 2p+k\}$, $\{e_0\}$ be an orthonormal basis of $\mathfrak{s}$ corresponding to the decomposition $\mathfrak{s}=\mathfrak{n}\oplus\mathfrak{z}\oplus\R$. We denote by $E_{l}$ the left-invariant vector field on $S$ determined by $e_{l}$, $0\leq l\leq 2p+k$. Damek \cite[Theorem 2.1]{D1} (see also \cite[P.234]{DR}) showed that the Laplace-Beltrami operator $\mathcal{L}$ associated to the left-invariant metric $\langle,\rangle_S$ has the form
\begin{equation*}
\mathcal{L}=\sum_{l=0}^{2p+k}E_{l}^2-QE_0.
\end{equation*}
Let $\partial_i,\:\partial_r,\:\partial_a$ be the partial derivatives for the system of coordinates $(X_i,Z_r,a)$ corresponding to $(e_i,e_r,e_0)$. Applying the definition of vector fields: 
$$E_{l}f(X,Z,a)=\frac{d}{dt}f\left((X,Z,a)\exp(tE_{l})\right)|_{t=0},$$
one can show that
\begin{equation*}
E_0=a\partial_a,\:\:E_i=a\partial_i+\frac{a}{2}\sum_{r=2p+1}^{2p+k}\langle[X,e_i],e_r\rangle\partial_r,\:\:E_r=a^2\partial_r,\:\:1\leq i\leq 2p,\:\:2p+1\leq r\leq 2p+k.
\end{equation*}
Using these expressions $\mathcal{L}$ can be written as \cite[P.234]{DR}
\begin{equation}\label{laplacebeltramionna}
\mathcal{L}=a^2\partial_a^2+\mathcal{L}_a+(1-Q)a\partial_a,
\end{equation}
where \begin{equation}\label{la}
\mathcal{L}_a=a(a+\frac{1}{4}\|X\|^2)\sum_{r=2p+1}^{2p+k}\partial_r^2+a\sum_{i=1}^{2p}\partial_i^2+a^2\sum_{2p+1}^{2p+k}\sum_{i=1}^{2p}\langle[X,e_i],e_r\rangle\partial_r\partial_i.
\end{equation}
The formula for the Poisson kernel $\mathcal{P}:S\times N\to(0,\infty)$, corresponding to $\mathcal{L}$ is given by \cite[Theorem 2.2]{D1} (see also \cite[P.409]{ACD}) 
\begin{equation*}
\mathcal{P}(x,n)=P_a(n_1^{-1}n),\:\:\:\:\:x=(n_1,a)\in S,\:\:n\in N,
\end{equation*}
where $P$ is the function on $N$ defined by
\begin{equation}\label{poissonone}
P(n)=P(X,Z)=\frac{c_{p,k}}{\left(\left(1+\frac{\|X\|^2}{4}\right)^2+\|Z\|^2\right)^Q},\:\:\:\:n=(X,Z)\in N,
\end{equation}
and $c_{p,k}$ is a positive constant so that $\|P\|_{L^1(N)}=1$. Using the notion of dilation of a function (\ref{dilationoffunction}) we get that
\begin{equation}\label{poissonna}
P_a(n)=P_a(X,Z)=\frac{c_{p,k}\:a^Q}{\left(\left(a+\frac{\|X\|^2}{4}\right)^2+\|Z\|^2\right)^Q},\:\:\:\:n=(X,Z)\in N,\:\:a>0.
\end{equation} 
Remark \ref{approximateidentity} implies that $\{P_a\mid a\in A\}$, is an approximate identity on $N$. Expanding the square involving $\|X\|$ and $a$ in the denominator of the right-hand side of (\ref{poissonna}), and then making use of the expression  (\ref{cannonicalnorm}) of $d$, we obtain the following alternative formula of the Poisson kernel.
\begin{equation}\label{alternativepoisson}
P_a(n)=P_a(X,Z)=\frac{16^Qc_{p,k}\:a^Q}{\left(16a^2+8a\|X\|^2+d(X,Z)^2\right)^Q},\:\:\:\:n=(X,Z)\in N,\:\:a>0.
\end{equation}
\begin{remark}\label{propertiesofpoisson}
We list down the following properties of the function $P_a$ which can be derived from (\ref{alternativepoisson}) and (\ref{polarcordinatena}).
\begin{enumerate}
	\item[i)] $P_a(n)=P_a(n^{-1})$,\:\:for all\:\:$n\in N,\:a>0$.
	\item[ii)] For $1\leq r\leq\infty,\:P_a\in L^r(N)$,\:\:for all\:\:$a>0$.
\end{enumerate}
\end{remark}
It turns out that more general eigenfunctions of $\mathcal{L}$ can be obtained by considering the complex power of the Poisson kernel. For $\lambda\in\C$, the $\lambda$-Poisson kernel is defined as
\begin{equation}\label{lambdapoisson}
\mathcal{P}_{\lambda}(x,n)=\left[\frac{\mathcal{P}(x,n)}{P(\underline{0})}\right]^{\frac{1}{2}-\frac{i\lambda}{Q}}=\left[\frac{P_a(n_1^{-1}n)}{P(\underline{0})}\right]^{\frac{1}{2}-\frac{i\lambda}{Q}},\:\:\:\:x=(n_1,a)\in S,\:\:n\in N.
\end{equation}
We note from the expression of the function $P$ given in (\ref{poissonone}) that $P(\underline{0})=c_{p,k}$. 

It is well-known that for $\lambda\in\C$, the function $\mathcal{P}_{\lambda}(.,n)$ is an eigenfunction of $\mathcal{L}$ with eigenvalue $-(\lambda^2+\rho^2)$, for each fixed $n\in N$, \cite[P.654]{ADY}. We note from above that $\mathcal{P}_{i\rho}=\mathcal{P}$ and hence  $\{\mathcal{P}_{i\rho}((\underline{0},a),\cdot)\mid a\in A\}=\{P_a\mid a\in A\}$, is an approximate identity on $N$. 
We observe from from the expression of $P_a$ given in (\ref{alternativepoisson}) and polar form of the integration on $N$ (see \ref{polarcordinatena}) that if  $\text{Im}(\lambda)\in (0,\infty)$ then $\mathcal{P}_{\lambda}(x,.)\in L^r(N)$, $r\in [1,\infty]$, for each $x\in S$.
We also have the following important formula \cite[Lemma 2.3]{PK},
\begin{equation}\label{intplamda}
\int_{N}\mathcal{P}_{\lambda}(e,n)\:dm(n)=\int_{N}\left[\frac{P(n)}{P(\underline{0})}\right]^{\frac{1}{2}-\frac{i\lambda}{Q}}\:dm(n)=\frac{{\bf c}(-\lambda)}{c_{p,k}},\:\:\:\text{Im}(\lambda)>0,
\end{equation}
where $\bf c (\lambda)$ generalizes Harish-Chandra $\bf c$-function and is given by
\begin{equation*}
{\bf c}(\lambda)=\frac{2^{Q-2i\lambda}\Gamma(2i\lambda)\Gamma\left(\frac{2p+k+1}{2}\right)}{\Gamma\left(\frac{Q}{2}+i\lambda\right)\Gamma\left(\frac{p+1}{2}+i\lambda\right)},\:\:\:\text{Im}(\lambda)<0.
\end{equation*}
From the above formula it follows that the $\bf c$-function has no pole or zero in $\{\lambda\in \C:\text{Im}(\lambda)<0\}$. Therefore, using (\ref{intplamda}), we can normalize $\mathcal P_{\lambda}$, for $\text{Im}(\lambda)>0$, to define
\begin{equation}\label{normalisedpoisson}
\tilde{\mathcal{P}_{\lambda}}(x,n)=C_{\lambda}\mathcal{P}_{\lambda}(x,n),\:\:\:x\in S,\:n\in N,
\end{equation}
where $C_{\lambda}=c_{p,k}{\bf c}(-\lambda)^{-1}$. For $\text{Im}(\lambda)>0$, the $\lambda$-Poisson transform of a measure $\mu$ on $N$ is defined by
\begin{equation}\label{plambdamu}
\mathcal{P}_{\lambda}[\mu](n,a)=\int_{N}\tilde{\mathcal{P}_{\lambda}}((n,a),n')\:d\mu(n'),
\end{equation}
whenever the integral converges absolutely for every $(n,a)\in S$. In this case, we say that the $\lambda$-Poisson transform $\mathcal{P}_{\lambda}[\mu]$ of $\mu$ is well-defined. If $d\mu=fdm$ for some $f\in L^r(N)$, where $r\in[1,\infty]$, then $\mathcal{P}_{\lambda}[\mu]$ is well-defined and we denote it by $\mathcal{P}_{\lambda}f$. Since for each $n\in N$, $\tilde{\mathcal{P}_{\lambda}}(.,n)$ is an eigenfunction of $\mathcal{L}$ with eigenvalue $-(\lambda^2+\rho^2)$, it follows that $\mathcal{P}_{\lambda}[\mu]$ is also an eigenfunction of $\mathcal{L}$ with the same eigenvalue provided $\mathcal{P}_{\lambda}[\mu]$ is well-defined. Using the definition of $\mathcal{P}_{\lambda}$ given in (\ref{lambdapoisson}) and the relation (\ref{normalisedpoisson}), we make the following important observation for $\text{Im}(\lambda)>0$, $x=(n_1,a)\in S$ and $n\in N$.
\begin{eqnarray}\label{relationbetweenpandq}
\tilde{\mathcal{P}_{\lambda}}(x,n)&=&C_{\lambda}\mathcal{P}_{\lambda}(x,n)\nonumber\\&=&C_{\lambda}\left[P(\underline{0})^{-1}P_a(n_1^{-1}n)\right]^{\frac{1}{2}-\frac{i\lambda}{Q}}\nonumber
\\&=&C_{\lambda}\left[P(\underline{0})^{-1}a^{-Q}P\left(\delta_{a^{-1}}(n_1^{-1}n)\right)\right]^{\frac{1}{2}-\frac{i\lambda}{Q}}\nonumber\\
&=&C_{\lambda}a^{\frac{Q}{2}+i\lambda}a^{-Q}\left[P(\underline{0})^{-1}P\left(\delta_{a^{-1}}(n_1^{-1}n)\right)\right]^{\frac{1}{2}-\frac{i\lambda}{Q}}\nonumber\\
&=& a^{\frac{Q}{2}+i\lambda}a^{-Q}q^{\lambda}\left(\delta_{a^{-1}}(n_1^{-1}n)\right)\nonumber\\
&=& a^{\frac{Q}{2}+i\lambda}q^{\lambda}_a(n_1^{-1}n),
\end{eqnarray}
where
\begin{equation}\label{qlambda}
q^{\lambda}(n):=C_{\lambda}\left[P(\underline{0})^{-1}P(n)\right]^{\frac{1}{2}-\frac{i\lambda}{Q}},\:\:\: q^{\lambda}_a(n)=a^{-Q}q^{\lambda}\left(\delta_{a^{-1}}(n)\right),
\end{equation}
$\text{Im}(\lambda)>0$, $n\in N$ and $a\in A$. We will need more explicit expression of the function $q_a^{\lambda}$, $\text{Im}(\lambda)>0$, which can be obtained from the expression (\ref{alternativepoisson}) of the function $P_a$
\begin{equation}\label{qbetaexpression}
q^{\lambda}_a(n)=c_{\lambda}\:\frac{a^{-2i\lambda}}{\left(16a^2+8a\|X\|^2+d(n)^2\right)^{\frac{Q}{2}-i\lambda}},\:\:\:n=(X,Z)\in N,\:a\in A,
\end{equation}
where $c_{\lambda}=16^{\rho-i\lambda}c_{p,k}{\bf c}(-\lambda )^{-1}$.
It is clear from the definition of $q^{\lambda}$ that if $\text{Im}(\lambda)>0$, then $q^{\lambda}\in L^r(N)$, for all $r\in[1,\infty]$ (as $P$ is so), and that 
\begin{equation*}
\int_{N}q^{\lambda}(n)\:dm(n)=1. 
\end{equation*}
It thus follow that $\{q^{\lambda}_a\mid a\in A\}$, is an approximate identity on $N$. For a measure $\mu$ on $N$ and $\text{Im}(\lambda)>0$, we define the convolution integral
\begin{equation}\label{qlamconvmu}
\mathcal{Q}_{\lambda}[\mu](n,a):=\mu\ast q^{\lambda}_a(n)=a^{-Q}\int_{N}q^{\lambda}\left(\delta_{a^{-1}}(n_1^{-1}n)\right)\:d\mu(n_1),
\end{equation}
whenever the integral converges absolutely for every $(n,a)\in S$ and say that $\mathcal{Q}_{\lambda}[\mu]$ is well-defined. If $d\mu=fdm$ for some $f\in L^r(N)$, where $r\in[1,\infty]$, then $\mathcal{Q}_{\lambda}[\mu]$ is well-defined and we denote it by $\mathcal{Q}_{\lambda}f$. From the definition of the $\lambda$-Poisson integral (\ref{plambdamu}) and (\ref{relationbetweenpandq}), it follows that for a measure $\mu$ with well-defined $\mathcal{P}_{\lambda}[\mu]$ with $\text{Im}(\lambda)>0$,
\begin{equation}\label{plamdaandqlamda}
\mathcal{P}_{\lambda}[\mu](n,a)=a^{\frac{Q}{2}+i\lambda}\mathcal{Q}_{\lambda}[\mu](n,a),\:\:\:\text{for all}\:\:(n,a)\in S.
\end{equation}
\begin{remark}\label{qlambdaf}
Since $\{q^{\lambda}_a\mid a\in A\}$, is an approximate identity on $N$, it follows from (\ref{plamdaandqlamda}) that for $f\in C_c(N)$, $\text{Im}(\lambda)>0$,
\begin{equation*}
\lim_{a\to 0}\frac{\mathcal{P}_{\lambda}f(n,a)}{a^{\frac{Q}{2}+i\lambda}}=\lim_{a\to 0}\mathcal{Q}_{\lambda}f(n,a)=f(n),
\end{equation*}uniformly for $n\in N$.  
\end{remark}
In this paper we will be interested only in the case $\lambda=i\beta$, for $\beta>0$. Using the formula of $P_a$ given in (\ref{alternativepoisson}) we can explicitly write down the expression of  $\mathcal{P}_{i\beta}[\mu]$ for a suitable measure $\mu$. In this regard, we fix some $\beta>0$, and denote by $M_{\beta}$ the set of all measures $\mu$ on $N$ such that $\mathcal{P}_{i\beta}[\mu]$ (equivalently $\mathcal{Q}_{i\beta}[\mu]$) is well-defined. If $\mu\in M_{\beta}$ then for all $n=(X,Z)\in N$, $a\in A$,
\begin{eqnarray}
\mathcal{Q}_{i\beta}[\mu](n,a)&=&C_{i\beta}\:a^{-Q}\int_{N}\left[P(\underline{0})^{-1}P_1\left(\delta_{a^{-1}}(n_1^{-1}n)\right)\right]^{\frac{1}{2}+\frac{\beta}{Q}}\:d\mu(n_1)\nonumber\\
&=&C_{i\beta}\:a^{-Q}\int_{N}\frac{16^{Q\left(\frac{1}{2}+\frac{\beta}{Q}\right)}}{\left(16+8\frac{\|X-X_1\|^2}{a}+\frac{d\left((X_1,Z_1)^{-1}(X,Z)\right)^2}{a^2}\right)^{Q\left(\frac{1}{2}+\frac{\beta}{Q}\right)}}\:d\mu(X_1,Z_1)\nonumber\\
&=&c_{\beta}\:a^{2\beta}\int_{N}\frac{1}{\left(16a^2+8a\|X-X_1\|^2+d\left((X_1,Z_1)^{-1}(X,Z)\right)^2\right)^{\rho+\beta}}\:d\mu(X_1,Z_1)\label{explicitqbetamu},
\end{eqnarray}
where 
\begin{equation*}
c_{\beta}=16^{\rho+\beta}C_{i\beta}=16^{\rho+\beta}c_{p,k}{\bf c}(-i\beta)^{-1}>0. 
\end{equation*}
The following elementary lemma gives a lower bound for the Poisson kernel $P_a(n)$, for large values of $d(n)$.  
\begin{lemma}\label{ratioest}
If $R\geq 1$, then for each $a>0$, there exists a constant $C_a>0$, such that
\begin{equation*}
\frac{1}{16a^2+d(n)^2+8a\|X\|^2}\geq \frac{C_a}{{16a^2+\frac{d(n)^2}{4\tau^2}}},\:\:\:\:\text{for all}\:\:\:n=(X,Z)\in  B(\underline{0},2R)^c.
\end{equation*}
\end{lemma}
\begin{proof}
We take $a>0$, and consider the quotient
\begin{equation*}
T(n)=\frac{16a^2+d(n)^2+8a\|X\|^2}{16a^2+\frac{d(n)^2}{4\tau^2}},\:\:\:n\in N.
\end{equation*}
For $n=(X,Z)\in N$ with $d(n)>2R$ we have  
\begin{equation*}
d(n)^2\geq\|X\|^4\geq\|X\|^2,\:\:\:\:\:\:\:\:\text{if $\|X\|>1$},
\end{equation*}
and hence
\begin{equation*}
T(n)\leq 4\tau^2\left(\frac{16a^2}{d(n)^2}+1+8a\frac{\|X\|^2}{d(n)^2}\right)\leq 4\tau^2\left(\frac{16a^2}{4R^2}+1+8a\right).
\end{equation*}
On the other hand, if $\|X\|\leq 1$, then 
\begin{equation*}
T(n)\leq 4\tau^2\left(\frac{16a^2}{d(n)^2}+1+8a\frac{\|X\|^2}{d(n)^2}\right)\leq 4\tau^2\left(\frac{4a^2}{R^2}+1+\frac{8a}{4R^2}\right).
\end{equation*}
As $R\geq 1$, combinig both the inequalities we get that for all $n=(X,Z)\in B(\underline{0},2R)^c $
\begin{equation}\label{ratioineq}
T(n)\leq 4\tau^2(4a^2+1+8a).
\end{equation}
The result follows by denoting $C_a=(4\tau^2(4a^2+1+8a))^{-1}$. 
\end{proof}
\begin{lemma}\label{integralfinite}
Let $\beta>0$ and $\mu$ be a positive measure on $N$ such that $\mu\ast q^{i\beta}_a(\underline{0})$ is finite for some $a\in A$. Then 
\begin{equation*}
\int_N\left(16a^2+\frac{d(n)^2}{4\tau^2}\right)^{-\beta-\rho}\:d\mu(n)<\infty.
\end{equation*}	
\end{lemma}
\begin{proof}
Since the integrand is a continuous function on $N$, it suffices to show that 
\begin{equation*}
\int_{B(\underline{0},2)^c}\left(16a^2+\frac{d(n)^2}{4\tau^2}\right)^{-\beta-\rho}\:d\mu(n)<\infty.
\end{equation*}
Using Lemma \ref{ratioest} and the explicit expression of $q_a^{\lambda}$, for $\lambda=i\beta$ (see \ref{qbetaexpression})), we obtain
\begin{eqnarray*}
&&a^{2\beta}\int_{B(\underline{0},2)^c}\left(16a^2+\frac{d(n)^2}{4\tau^2}\right)^{-\beta-\rho}\:d\mu(n)\\
&\leq& C_{a,\beta}\:a^{2\beta}\int_{B(\underline{0},2)^c}\left(16a^2+d(n)^2+8a\|X\|^2\right)^{-\beta-\rho}\:d\mu(n)\\
&=& C_{a,\beta}^{\prime}\:\mu\ast q^{i\beta}_a(0)<\infty. 
	\end{eqnarray*}
This completes the proof.
\end{proof}
We have observed in Remark \ref{qlambdaf} that $\mathcal{Q}_{i\beta}f(\cdot, a)\to f$, as $a\to 0$, uniformly on $N$, whenever $f\in C_c(N)$ and $\beta>0$. However, a stronger result is true. 
\begin{lemma}\label{unifna}
If $\beta>0$ and $f\in C_c(N)$, then
\begin{equation*}
\frac{\mathcal{Q}_{i\beta}f(\cdot,a)}{q^{i\beta}}\to \frac{f}{q^{i\beta}},
\end{equation*}
uniformly on $N$ as $a\to 0$.
\end{lemma}
\begin{proof}
We assume that $\text{supp}\:f\subset B(\underline{0},R)$ for some $R>1$. 
Since $q^{i\beta}$ is a strictly positive continuous function on $N$, it follows that $\frac{1}{q^{i\beta}}$ is bounded in $\overline{B(\underline{0},2\tau R)}$. In view of Remark \ref{qlambdaf}, it suffices to prove that  
\begin{equation*}
\frac{\mathcal{Q}_{i\beta}f(n,a)}{q^{i\beta}(n)}\to 0,
\end{equation*}
uniformly for $n\in B(\underline{0},2\tau R)^c$, as $a\to 0$.
From the quasi-triangle inequality (\ref{quasinorm}) for $d$, we have 
\begin{equation}\label{reversetriangle}
d(n_1^{-1}n)\geq \frac{d(n)}{\tau}-d(n_1),\:\:\:\text{for all $n\in N$, $n_1\in N$}.
\end{equation}
For+ $n\in B(\underline{0},2\tau R)^c$ and $n_1\in B(\underline{0},R)$
\begin{equation*}
d(n_1)<R\leq\frac{d(n)}{2\tau}
\end{equation*}
Hence, for $n\in B(\underline{0},2\tau R)^c$ and $n_1\in B(\underline{0},R)$,
\begin{equation}\label{normcomp}
d(n_1^{-1}n)\geq \frac{d(n)}{\tau}-\frac{d(n)}{2\tau}=\frac{d(n)}{2\tau}.
\end{equation}
From the expression of $\mathcal Q_{i\beta}f$ (see \ref{explicitqbetamu}) and the inequality (\ref{normcomp}) above, it follows that for $n=(X,Z)\in B(\underline{0},2\tau R)^c$, $a>0$,
\begin{eqnarray*}
\left|\mathcal{Q}_{i\beta}f(n,a)\right|&\leq & c_{\beta}\:a^{2\beta}\int_{B(\underline{0},R)}\frac{|f(X_1,Z_1)|}{\left(16a^2+8a\|X-X_1\|^2+d\left((X_1,Z_1)^{-1}(X,Z)\right)^2\right)^{\rho+\beta}}\:dX_1dZ_1\nonumber\\
&\leq&c_{\beta}\:a^{2\beta}\int_{B(\underline{0},R)}\frac{|f(X_1,Z_1)|}{\left(16a^2+8a\|X-X_1\|^2+\frac{d(X,Z)^2}{4\tau^2}\right)^{\beta+\rho}}\:dX_1dZ_1\\
&\leq&c_{\beta}\:a^{2\beta}\int_{B(\underline{0},R)}\frac{|f(X_1,Z_1)|}{\left(16a^2+\frac{d(X,Z)^2}{4\tau^2}\right)^{\beta+\rho}}\:dX_1dZ_1\nonumber\\
\end{eqnarray*}
Using the expression of $q^{i\beta}$ given in (\ref{qbetaexpression}) and Lemma \ref{ratioest} for $a=1$ (as $\tau\geq 1$), it follows from the inequality above that for all $n=(X,Z)\in B(\underline{0},2\tau R)^c$, 
\begin{eqnarray*}
\frac{|\mathcal{Q}_{i\beta}f(n, a)|}{q^{i\beta}(n)}&\leq& c_{\beta}\:a^{2\beta}\left(\frac{16+8\|X\|^2+d(n)^2}{16a^2+\frac{d(n)^2}{4\tau^2}}\right)^{\beta+\rho}\int_{B(\underline{0},R)}|f(n_1)|\:dm(n_1)\\ 
&\leq&c_{\beta,\tau}\:a^{2\beta}\left(\frac{16+\frac{d(n)^2}{4\tau^2}}{16a^2+\frac{d(n)^2}{4\tau^2}}\right)^{\beta+\rho}\int_{B(\underline{0},R)}|f(n_1)|\:dm(n_1)\\
&\leq&c_{\beta,\tau}\:a^{2\beta}\left(\frac{64\tau^2}{d(n)^2}+1\right)^{\beta+\rho}\|f\|_{L^1(N)}\\
&\leq& C_{R}\:a^{2\beta}\|f\|_{L^1(N)},
\end{eqnarray*} 
as $d(n)^2>2\tau R$. Letting $a\to 0$ in the last inequality, we complete the proof.
\end{proof}
Recall that for $\beta>0$, $M_{\beta}$ denotes the set of all measures $\mu$ on $N$ such that $\mathcal{Q}_{i\beta}[\mu]$ is well-defined (see the paragraph after Remark \ref{qlambdaf}). It is clear that $L^r(N)\subset M_{\beta}$, for all $r\in[1,\infty]$. We also note that if $|\mu|(N)$ is finite, then $\mu\in M_{\beta}$. In particular, every complex measure $\mu$ on $N$ belongs to $M_{\beta}$. We have the following observation regarding this class of measures. 
\begin{lemma}\label{fubinina}
Suppose $\beta>0$. If $\nu\in M_{\beta}$, and $f\in C_c(N)$, then for each fixed $a\in A$,
\begin{equation*}
\int_{N}\mathcal{Q}_{i\beta}f(n,a)\:d\nu(n)=\int_{N}\mathcal{Q}_{i\beta}[\nu](n,a)f(n)\:dm(n).
\end{equation*}
\end{lemma}
\begin{proof}
The result will follow by interchanging integrals using Fubini's theorem. In order to apply Fubini's theorem we have to show that 
\begin{equation*}
\int_{N}\int_{\text{supp}\:f}q^{i\beta}_a(n_1^{-1}n)|f(n_1)|\:dm(n_1)\:d|\nu|(n)<\infty.
\end{equation*}
We asuume that $\text{supp}\:f\subset B(\underline{0},R)$, for some $R>1$. Then for each fixed $a\in A$,
\begin{eqnarray*}
I&=&\int_{N}\int_{B(\underline{0},R)}q^{i\beta}_a(n_1^{-1}n)|f(n_1)|\:dm(n_1)\:d|\nu|(n)\\
&=&c_{\beta}\:a^{2\beta}\int_{B(\underline{0},2\tau R)}\int_{B(\underline{0},R)}\frac{|f(X_1,Z_1)|\:dm(X_1,Z_1)\:d|\nu|(n)}{\left(16a^2+8a\|X-X_1\|^2+d\left((X_1,Z_1)^{-1}(X,Z)\right)^2\right)^{\rho+\beta}}\\&&\:+c_{\beta}\:a^{2\beta}\int_{B(\underline{0},2\tau R)^c}\int_{B(\underline{0},R)}\frac{|f(X_1,Z_1)|\:dm(X_1,Z_1)\:d|\nu|(n)}{\left(16a^2+8a\|X-X_1\|^2+d\left((X_1,Z_1)^{-1}(X,Z)\right)^2\right)^{\rho+\beta}}\\
&\leq&c_{\beta}\:a^{2\beta}(16a^2)^{-\beta-\rho}|\nu|(B(\underline{0},2\tau R))\|f\|_{L^1(N)}\\&&
\:\:\:+c_{\beta}\:a^{2\beta}\int_{B(\underline{0},2\tau R)^c}\int_{B(\underline{0},R)}\frac{|f(X_1,Z_1)|}{\left(16a^2+8a\|X-X_1\|^2+\frac{d(X,Z)^2}{4\tau^2}\right)^{\beta+\rho}}\:dm(X_1,Z_1)\:d|\nu|(n)\\&&\:\:\:\:\:\:\:\:\:(\text{using (\ref{normcomp}) in the second integral})\\&\leq&c_{\beta}\:a^{2\beta}(16a^2)^{-\beta-\rho}|\nu|(B(\underline{0},2\tau R))\|f\|_{L^1(N)}\\&&\:\:\:+c_{\beta}\:a^{2\beta}\|f\|_{L^1(N)}\int_{ B(\underline{0},2\tau  R)^c}\left(16a^2+\frac{d(X,Z)^2}{4\tau^2}\right)^{-\beta-\rho}\:d|\nu|(n).
\end{eqnarray*}
As $\nu\in M_{\beta}$ we have finiteness of the quantity $|\nu|\ast q^{i\beta}_a(\underline{0})$. Lemma \ref{integralfinite} now implies that the integral on the right-hand side of the inequality above is finite. This completes the proof.
\end{proof}
We end this section with the following important definitions which constitute the heart of the matter.
\begin{definition}\label{impdefnna}
\begin{enumerate}
\item[i)] A function $u$ defined on $S$ is said to have admissible limit $L\in\C$, at $n_0\in N$, if for each $\alpha>0$
\begin{equation*}
\lim_{\substack{a\to 0\\(n,a)\in \Gamma_{\alpha}(n_0)}}u(n,a)=L,
\end{equation*}
where  
\begin{equation}\label{admissibledomain}
\Gamma_{\alpha}(n_0):=\{(n,a)\in S\mid{\bf d}(n_0,n)<\alpha a\}=\{(n,a)\in S\mid d(n_0^{-1}n)<\alpha a\}
\end{equation} 
is called the admissible domain with vertex at $n_0$ and aperture $\alpha$.
\item[ii)] Given a measure $\mu$ on $N$, we say that $\mu$ has strong derivative $L\in[0,\infty)$ at $n_0\in N$ if
\begin{equation*}
\lim_{r\to 0}\frac{\mu(n_0\delta_r(B))}{m(n_0\delta_r(B))}=L,
\end{equation*}
holds for every $d$-ball $B\subset N$. The strong derivative of $\mu$ at $n_0$, if it exists, is denoted by $D\mu(n_0)$. 
\item[iii)] A sequence of complex valued functions $\{u_j\}_{j\in\N}$ defined on $S$ is said to converge normally to a function $u$ if $\{u_j\}$ converges to $u$ uniformly on all compact subsets of $S$.
\item[iv)] A sequence of complex valued functions $\{u_j\}_{j\in\N}$ defined on $S$ is said to be locally bounded if given any compact set $K\subset S$, there exists a positive constant $C_K$ (depending only on $K$) such that for all $j\in\N$ and all $x\in K$,
\begin{equation*}
|u_j(x)|\leq C_K.
\end{equation*}
\item[v)] A sequence of positive measures $\{\mu_j\}$ on $N$ is said to converge to a positive measure $\mu$ on $N$ in weak* if
\begin{equation*}
\lim_{j\to\infty}\int_{N}\psi(n)\:d\mu_j(n)=\int_{N}\psi(n)\:d\mu(n),
\end{equation*}
for all $\psi\in C_c(N)$.
\item [vi)] For a differential operator $D$ on $S$, a smooth function $u$ on $S$ satisfying $Du=0$ is said to be a $D$-harmonic function. 
\end{enumerate}
\end{definition}
\begin{remark}
For the notion of admissible convergence in the context of Riemannian symmetric spaces of noncompact type with real rank one we refer the reader to \cite[P.158]{KP}. We recall that for a measure $\mu$ on $N$, we say that $n_0\in N$ is a Lebesgue point of $\mu$ if there exists $L\in\C$, such that
\begin{equation}\label{lebn}
\lim_{r\to 0}\frac{|\mu-Lm|(B(n_0,r))}{m(B(n_0,r))}=0.
\end{equation}
It can be seen that $\mu$ has strong derivative $L$ at each Lebesgue point $n_0$ (see \cite[Remark 3.2]{Sar}). It was shown in \cite[Theorem 2.2]{Mi} that if $S=NA$ is a rank one symmetric space then $\mathcal{Q}_{\lambda}[\mu]$ (see (\ref{qlamconvmu})) has admissible limit $L$ at each Lebesgue point $n_0$ of $\mu$, where $L$ is as in (\ref{lebn}), whenever $\text{Im}(\lambda)>0$. In \cite[P.12-16]{Sar}, we have constructed an example of a positive measure on the Heisenberg group to show that the set of all points where strong derivative exists is strictly larger than the set of Lebesgue points of the measure. We shall show that if $\lambda=i\beta$, for some $\beta>0$, then $\mathcal{Q}_{\lambda}[\mu]$ has admissible limit $L$ at $n_0$ if and only if $\mu$ has strong derivative $L$ at $n_0$, provided $\mu$ is positive. 
\end{remark}

\section{Some auxilary results}
We start with the following result which relates normal convergence of Poisson integrals of a sequence of positive measures $\{\mu_j\mid j\in\N\}$ with the weak* convergence of $\{\mu_j\mid j\in\N\}$.
\begin{lemma}\label{normalna}
Let $\beta>0$. Suppose that $\{\mu_j\mid j\in\N\}\subset M_{\beta}$ and $\mu\in M_{\beta}$ are positive measures. If $\{\mathcal{Q}_{i\beta}[\mu_j]\}$ converges normally to $\mathcal{Q}_{i\beta}[\mu]$, then $\{\mu_j\}$ converges to $\mu$ in weak*.
\end{lemma}
\begin{proof}
We have to show that if $f\in C_c(N)$, then  
\begin{equation*}
\lim_{j\to\infty}\int_{N}f(n)\:d\mu_j(n)=\int_{N}f(n)\:d\mu(n).
\end{equation*}
We assume that $f\in C_c(N)$, with $\text{supp}\:f\subset B(\underline{0},R)$ and $R>1$. For any $a>0$, we write
\begin{eqnarray}
&&\int_{N}f(n)\:d\mu_j(n)-\int_{N}f(n)\:d\mu(n)\nonumber\\
&=&\int_{N}(f(n)-\mathcal{Q}_{i\beta}f(n,a))\:d\mu_j(n)+\int_{N}\mathcal{Q}_{i\beta} f(n,a)\:d\mu_j(n)-\int_{N}\mathcal{Q}_{i\beta}f(n,a)\:d\mu(n)\nonumber\\
&&\:\:\:\:\:+\int_{N}(\mathcal{Q}_{i\beta}f(n,a)-f(n))\:d\mu(n).\label{ineq1na}
\end{eqnarray}
Given $\epsilon>0$, by Lemma \ref{unifna} we get some $a_0>0$, such that for all $n\in N$
\begin{equation}\label{uniformineqna}
\frac{|\mathcal{Q}_{i\beta} f(n,a_0)-f(n)|}{q^{i\beta}(n)}<\epsilon.
\end{equation}
Using Lemma \ref{fubinina}, it follows that
\begin{equation*}
\int_{N}\mathcal{Q}_{i\beta}f(n,a)\:d\mu_j(n)-\int_{N}\mathcal{Q}_{i\beta}f(n,a)\:d\mu(n)=\int_{N}\left(\mathcal{Q}_{i\beta}[\mu_j](n,a)-\mathcal{Q}_{i\beta}[\mu](n,a)\right)f(n)\:dm(n).
\end{equation*}
Applying the relation above for $a=a_0$ in (\ref{ineq1na}) we get that
\begin{eqnarray}\label{estimatenormal}
&&\left|\int_{N}f(n)\:d\mu_j(n)-\int_{N}f(n)\:d\mu(n)\right|\nonumber\\
&\leq&\int_{N}\left|f(n)-\mathcal{Q}_{i\beta}f(n,a_0)\right|\:d\mu_j(n)\nonumber\\&&\:\:\:\:\:+\int_{B(\underline{0},R)}\left|\mathcal{Q}_{i\beta}[\mu_j](n,a_0)-
\mathcal{Q}_{i\beta}[\mu](n,a_0)\right||f(n)|\:dm(n)\nonumber\\&&\:\:\:\:\:\:+\int_{N}\left|f(n)-\mathcal{Q}_{i\beta}f(n,a_0)\right|\:d\mu(n)\nonumber\\&=&I_1(j)+I_2(j)+I_3.
\end{eqnarray}
In order to estimate $I_1(j)$ we use (\ref{uniformineqna}) to get 
\begin{equation*}
I_1(j)=\int_{N}\frac{\left|f(n)-\mathcal{Q}_{i\beta}f(n,a_0)\right|}{q^{i\beta}(n)}q^{i\beta}(n)\:d\mu_j(n)<\epsilon\int_{N}q^{i\beta}(n)\:d\mu_j(n)=\epsilon\mathcal{Q}_{i\beta}[\mu_j](e).
\end{equation*} 
Similarly, we can prove that
\begin{equation*}
I_3\leq\epsilon\mathcal{Q}_{i\beta}[\mu](e).
\end{equation*}
Since $\{\mathcal{Q}_{i\beta}[\mu_j]\}$ converges to $\mathcal{Q}_{i\beta}[\mu]$ normally, the sequence $\{\mathcal{Q}_{i\beta}[\mu_j](e)\}$, in particular, is bounded. Hence, taking $C$ to be the supremum of $\{\mathcal{Q}_{i\beta}[\mu_j](e)+\mathcal{Q}_{i\beta}[\mu](e)\}$, we get that for all $j\in\N$
\begin{equation*}
I_1(j)+I_3\leq2C\epsilon.
\end{equation*}
By hypothesis, $\{\mathcal{Q}_{i\beta}[\mu_j]\}$ converges normally to $\mathcal{Q}_{i\beta}[\mu]$. Hence, there exists some $j_0\in\N$ such that for all $j\geq j_0$,
\begin{equation*}
\|\mathcal{Q}_{i\beta}[\mu_j]-\mathcal{Q}_{i\beta}[\mu]\|_{L^{\infty}\left(\overline{B(\underline{0},R)}\times\{a_0\}\right)}<\epsilon.
\end{equation*}
This implies that for all $j\geq j_0$,
\begin{equation*}
I_2(j)\leq \epsilon \|f\|_{L^1(N)}.
\end{equation*}
Hence, It follows from (\ref{estimatenormal}) that for all $j\geq j_0$
\begin{equation*}
\left|\int_{N}f(n)\:d\mu_j(n)-\int_{N}f(n)\:d\mu(n)\right|\leq\epsilon (2C+\|f\|_{L^1(N)}).
\end{equation*}
This completes the prove.
\end{proof}
We will also need the following measure theoretic result on $N$ proved in \cite[Lemma 4.3]{Sar}. 
\begin{lemma}\label{mthna}
Suppose $\{\mu_j\}_{j\in\N}$, $\mu$  are positive measures on $N$ and $\{\mu_j\}$ converges to $\mu$ in weak*. Then for some $L\in[0,\infty)$, $\mu=Lm$ if and only if $\{\mu_j(B)\}$ converges to $Lm(B)$ for every $d$-ball $B\subset N$.
\end{lemma}
We shall next prove a result regarding pointwise comparison between the Hardy-Littlewood maximal function of a positive measure on $N$ and Poisson maximal functions of the same measure. We recall that for a positive measure $\mu$ on $N$, the Hardy-Littlewood maximal function $M_{HL}(\mu)$ of $\mu$ is defined by
\begin{equation*}
M_{HL}(\mu)(n)=\sup_{r>0}\frac{\mu(B(n,r))}{m(B(n,r))},\:\:\:\:n\in N.
\end{equation*}
\begin{lemma}\label{maximalna}
If $\mu\in M_{\beta}$ is a positive measure, for some $\beta>0$, and $\alpha>0$, then there exist positive constants $C_{\beta}$ and $C_{\alpha,\beta}$ such that for all $n_0\in N$,
\begin{equation*}
C_{\beta}M_{HL}(\mu)(n_0)\leq\sup_{a>0}\mathcal{Q}_{i\beta}[\mu](n_0,a)\leq\sup_{na\in \Gamma_{\alpha}(n_0)}\mathcal{Q}_{i\beta}[\mu](n,a)\leq C_{\alpha,\beta}M_{HL}(\mu)(n_0).
\end{equation*}
\end{lemma}
\begin{proof}
We fix an $n_0=(X_0,Z_0)\in N$ and note that the second inequality follows from the definition of supremum. To prove the left-most inequality we take $a>0$, and using the expression of $\mathcal Q_{i\beta}[\mu]$ given in (\ref{explicitqbetamu}), note that
\begin{eqnarray}
\mathcal{Q}_{i\beta}[\mu](n_0,a)&=&c_{\beta}\:a^{-Q}\int_{N}\frac{1}{\left(16+8\frac{\|X_0-X\|^2}{a}+\frac{d\left((X,Z)^{-1}(X_0,Z_0)\right)^2}{a^2}\right)^{\rho+\beta}}\:d\mu(X_1,Z_1)\nonumber\\
&\geq&c_{\beta}\:a^{-Q}\int_{B(n_0,a)}\frac{1}{\left(16+8\frac{\|X_0-X\|^2}{a}+\frac{d\left((X,Z)^{-1}(X_0,Z_0)\right)^2}{a^2}\right)^{\rho+\beta}}\:d\mu(X_1,Z_1)
\label{maximalna1}.
\end{eqnarray}
We observe that for $(X,Z)\in B(n_0,a)$, 
\begin{equation*}
d\left((X,Z)^{-1}(X_0,Z_0)\right)^2=\|X_0-X\|^4+16 \|Z-Z_0\|^2<a^2,
\end{equation*}
and hence  
\begin{equation*}
\|X_0-X\|^2<a.
\end{equation*}
Consequently, for all $(X,Z)\in B(n_0,a)$
\begin{equation*}
\frac{c_{\beta}}{\left(16+8\frac{\|X_0-X\|^2}{a}+\frac{d\left((X,Z)^{-1}(X_0,Z_0)\right)^2}{a^2}\right)^{\rho+\beta}}\geq\frac{c_{\beta}}{(16+8+1)^{\rho+\beta}}=C_{\beta}'. 
\end{equation*}
 Using this observation in (\ref{maximalna1}), we get
\begin{equation*}
\mathcal{Q}_{i\beta}[\mu](n_0,a)\geq C_{\beta}' a^{-Q} \mu(B(n_0,a))=C_{\beta}\frac{\mu(B(n_0,a))}{m(B(n_0,a))}.
\end{equation*}
Taking supremum over $a>0$, on both sides of the inequalty above, we get
\begin{equation}\label{poissonradial}
C_{\beta}M_{HL}(\mu)(n_0)\leq\sup_{a>0}\mathcal{Q}_{i\beta}[\mu](n_0,a).
\end{equation}
For the right-most inequality, we take $(n,a)\in \Gamma_{\alpha}(n_0)$, and write
\begin{eqnarray}\label{nontangentialmax}
\mathcal{Q}_{i\beta}[\mu](n,a)&=&\mu\ast q_a^{i\beta}(n)\nonumber\\&=&\int_{B(n,a\alpha)}q^{i\beta}_a(n_1^{-1}n)\:d\mu(n_1)+\int_{B(n,a\alpha)^c}q^{i\beta}_a(n_1^{-1}n)\:d\mu(n_1)\nonumber\\
&=& \int_{B(n,a\alpha)}q^{i\beta}_a(n_1^{-1}n)\:d\mu(n_1)\nonumber\\
&&\:\:\:\:\:\:+a^{-Q}\sum_{j=1}^{\infty}\int_{\{n_1\mid 2^{j-1}a\alpha\leq d(n^{-1}n_1)<2^ja\alpha\}}q^{i\beta}\left(\delta_{a^{-1}}(n_1^{-1}n)\right)\:d\mu(n_1)\nonumber\\&=&I+\sum_{j=1}^{\infty}I_j.\:\:\:\:\:
\end{eqnarray}
It follows from the expression of $q^{i\beta}_a$ (see (\ref{qbetaexpression})) that
\begin{equation*}
q^{i\beta}_a(n_1^{-1}n)\leq c_{\beta}\:a^{-Q},\:\:\:\text{for all}\:\:a>0.
\end{equation*}
This shows that
\begin{equation*}
I\leq c_{\beta}\:a^{-Q}\int_{B(n,a\alpha)}d\mu(n_1)=c_{\alpha,\beta}'\:\frac{\mu(B(n,a\alpha))}{m((B(n_0,a\alpha)))}.
\end{equation*}
It again follows from the expression of $q^{i\beta}$ that  
\begin{equation}\label{estimateofp1}
q^{i\beta}(n_2)\leq c_{\beta}\:d(n_2)^{-Q-2\beta},\:\:\text{for all}\:\:n_2\in N\setminus\{\underline{0}\}.\end{equation}
As $(n,a)\in \Gamma_{\alpha}(n_0)$, that is, $d(n_0^{-1}n)<a\alpha$, it follows from the quasi-triangle inequality (\ref{quasinorm}) that for all $n_1\in B(n,a\alpha)$
\begin{equation*}
d(n_0^{-1}n_1)\leq \tau(d(n_0^{-1}n)+d(n^{-1}n_1))< 2\tau a\alpha,
\end{equation*}
which in turn, implies that $B(n,a\alpha)\subset B(n_0,2\tau a\alpha)$. Applying this observation in the last inequality involving $I$, we get
\begin{equation*}\label{integralI}
I\leq c'_{\alpha,\beta}\:\frac{\mu(B(n_0,2\tau a\alpha))}{m((B(n_0,a\alpha)))}=C'_{\alpha,\beta}\:\frac{\mu(B(n_0,2\tau a\alpha))}{m(B(n_0,2\tau a\alpha))}\leq C_{\alpha,\beta}^{\prime}M_{HL}(\mu)(n_0).\end{equation*}
We now use the estimate (\ref{estimateofp1}) to obtain
\begin{eqnarray*}
I_j&\leq& c_{\beta}\:a^{-Q}\int_{\{n_1\mid 2^{j-1}a\alpha\leq d(n^{-1}n_1)<2^ja\alpha\}}d\left(\delta_{a^{-1}}(n_1^{-1}n)\right)^{-Q-2\beta}\:d\mu(n_1)\\&\leq&c_{\beta, \alpha}\:a^{-Q}\int_{\{n_1\mid 2^{j-1}a\alpha\leq d(n^{-1}n_1)<2^ja\alpha\}}2^{-(Q+2\beta)(j-1)}\:d\mu(n_1)\\&\leq&c_{\beta, \alpha}\:a^{-Q}2^{-(Q+2\beta)(j-1)}\mu(B(n,2^ja\alpha)).
\end{eqnarray*}
Since $d(n_0^{-1}n)<a\alpha$, we have as before $B((n,2^ja\alpha))\subset B(n_0,2^{j+1}\tau a\alpha)$. Hence,
\begin{eqnarray*}
I_j&\leq&c_{\beta,\alpha}\:a^{-Q}2^{-(Q+2\beta)(j-1)}\mu\left(B(n_0,2^{j+1}\tau a\alpha)\right)\\
&=&C_{\beta,\alpha}'\:2^{-(Q+2\beta)(j-1)}2^{(j+1)Q}\frac{\mu\left(B(n_0,2^{j+1}\tau a\alpha)\right)}{m\left(B(n_0,2^{j+1}\tau a\alpha)\right)}\\&\leq&C_{\beta,\alpha}'\:2^{-(Q+2\beta)(j-1)}2^{(j+1)Q}M_{HL}(\mu)(n_0)\\&=&C_{\alpha,\beta}''\:2^{-2j\beta}M_{HL}(\mu)(n_0).\end{eqnarray*}
Invoking these estimates of $I$ and $I_j$ in (\ref{nontangentialmax}), we obtain
\begin{equation*}
\mathcal{Q}_{i\beta}[\mu](n,a)\leq \left(C'_{\alpha,\beta}+C_{\alpha,\beta}''\sum_{j=1}^{\infty}2^{-2j\beta}\right)M_{HL}(\mu)(n_0).
\end{equation*}
Since $(n,a)\in \Gamma_{\alpha}(n_0)$ is arbitrary and $\sum_{j=1}^{\infty}2^{-2j\beta}$ is finite, we get the right-most inequality, namely there exists some positive constant $C_{\alpha,\beta}$ such that
\begin{equation*}
\sup_{na\in \Gamma_{\alpha}(n_0)}\mathcal{Q}_{i\beta}[\mu](n,a)\leq C_{\alpha,\beta}M_{HL}(\mu)(n_0).
\end{equation*}
This completes the proof.
\end{proof}
Given $\beta>0$, we define a second order differential operator $\mathcal{L}^{\beta}$ on $S$ (see \cite[Theorem 3.2]{DK}) having the same second order term as the Laplace-Beltrami operator $\mathcal{L}$ by the formula
\begin{equation*}
\mathcal{L}^{\beta}=a^2\partial_a^2+\mathcal{L}_a+(1-2\beta)a\partial_a,
\end{equation*}
where $\mathcal{L}$ is given by (\ref{laplacebeltramionna}). When $\beta=\rho=Q/2$, we recover $\mathcal{L}$. We note that 
\begin{equation}\label{diflandlbeta}
\mathcal{L}^{\beta}-\mathcal{L}=2(\rho-\beta)a\partial_a=2(\rho-\beta)E_0.
\end{equation}
We recall that $E_0=a\partial_a$ is the left-invariant vector field on $S$ corresponding to the basis element $e_0=(0,0,1)$ of $\mathfrak{s}$ and hence $\mathcal{L}^{\beta}$ is left $S$-invariant. The following lemma shows that there is a one to one correspondence between the eigenfunctions of $\mathcal{L}$ with eigenvalue $\beta^2-\rho^2$ and $\mathcal{L}^{\beta}$-harmonic functions (as defined in Definition \ref{impdefnna}, vi)).
\begin{lemma}\label{lbetaharmonic}
Let $\beta>0$ and let $u$ be a smooth function on $S$. Then $u$ is an eigenfunction of $\mathcal{L}$ with eigenvalue $\beta^2-\rho^2$, if and only if the function $(n,a)\mapsto a^{\beta-\rho}u(n,a)$, is $\mathcal{L}^{\beta}$-harmonic.
\end{lemma}
\begin{proof}
We set \begin{equation*}
F(n,a)=a^{\beta-\rho}u(n,a),\:\:\:(n,a)\in S.
\end{equation*} 
Suppose $u$ is an eigenfunction of $\mathcal{L}$ with eigenvalue $\beta^2-\rho^2$. 
Note that \begin{equation*}
\mathcal{L}(a^{\rho-\beta}F)=\mathcal{L}u=(\beta^2-\rho^2)u=(\beta^2-\rho^2)a^{\rho-\beta}F.
\end{equation*}
Since $\mathcal{L}_a$ does not have any term involving $\partial_a$ (see (\ref{la})), expanding the left-hand side of the equation above we obtain
\begin{eqnarray*}
	(\beta^2-\rho^2)a^{\rho-\beta}F&=&\left((1-2\rho)a\partial_a+a^2\partial_a^2+\mathcal{L}_a\right)(a^{\rho-\beta}F)\\&=&(1-2\rho)a\left((\rho-\beta)a^{\rho-\beta-1}F+a^{\rho-\beta}\partial_aF\right)\\&&\:\:\:+a^2\left((\rho-\beta)(\rho-\beta-1)a^{\rho-\beta-2}F+2(\rho-\beta)a^{\rho-\beta-1}\partial_aF+a^{\rho-\beta}\partial_a^2F\right)\\&&\:\:\:+a^{\rho-\beta}\mathcal{L}_aF\\
	&=&a^{\rho-\beta}\left((1-2\rho)(\rho-\beta)+(\rho-\beta)(\rho-\beta-1)\right)F\\
	&&\:\:\:+a^{\rho-\beta}\left((1-2\rho)+2(\rho-\beta)\right)a\partial_aF+a^{\rho-\beta}\left(a^2\partial_a^2F+\mathcal{L}_aF\right)\\
	&=&a^{\rho-\beta}\left(\beta^2-\rho^2+(1-2\beta)a\partial_a+a^2\partial_a^2+\mathcal{L}_a\right)F.
	\end{eqnarray*}
Canceling the term $(\beta^2-\rho^2)a^{\rho-\beta}F$ from both sides of the equation above shows that $F$ is $\mathcal{L}^{\beta}$-harmonic.

Conversely, suppose that $F$ is $\mathcal{L}^{\beta}$-harmonic. Using the definition of $\mathcal{L}^{\beta}$ and $F$, we can write
\begin{equation*}
\left(a^2\partial_a^2+\mathcal{L}_a+(1-2\beta)a\partial_a\right)(a^{\beta-\rho}u)=0.
\end{equation*}
Expanding the left-hand side of the equation above as before, we get
\begin{equation*}
(-\beta^2+\rho^2)u+\left((1-Q)a\partial_a+a^2\partial_a^2+\mathcal{L}_a\right)u=0.
\end{equation*}
Hence, $u$ is an eigenfunction of $\mathcal{L}$ with eigenvalue $\beta^2-\rho^2$.
\end{proof}
From (\ref{diflandlbeta}) we see that $\mathcal{L}^{\beta}$ and $\mathcal{L}$ differs by a first order term and hence $\mathcal{L}^{\beta}$ is an ellliptic operator. Applying a vast generalization of Montel's theorem, valid for solutions of hypoelliptic operators (and hence for $\mathcal{L}^{\beta}$-harmonic functions) proved in \cite[Theorem 4]{B} we get the following result.
\begin{lemma}\label{montelna}
Let $\beta>0$, and let $\{F_j\}$ be a sequence of $\mathcal{L}^{\beta}$-harmonic functions on $S$. If $\{F_j\}$ is locally bounded then it has a subsequence which converges normally to a $\mathcal{L}^{\beta}$-harmonic function $F$.
\end{lemma}
We recall that if $\beta>0$, and $\mu\in M_{\beta}$ is positive then $\mathcal{P}_{i\beta}[\mu]$ is a positive eigenfunction of the Laplace-Beltrami operator $\mathcal{L}$ with eigenvalue $\beta^2-\rho^2$. Characterization of such positive eigenfunctions was proved by Damek and Ricci in \cite[Theorem 7.11]{DR}.
\begin{lemma}\label{positiveeigen}
Suppose $u$ is a positive eigenfunction of the Laplace-Beltrami operator $\mathcal{L}$ on the Harmonic $NA$ group $S$ with eigenvalue $\beta^2-\rho^2$, for some $\beta>0$. Then there exists a unique positive measure $\mu$ on $N$ and a unique nonnegative constant $C$ such that
\begin{equation}\label{repofu}
u(n,a)=Ca^{\beta+\rho}+\mathcal{P}_{i\beta}[\mu](n,a),\:\:\:\:\text{for all}\:\:(n,a)\in S.
\end{equation}
In this case, the measure $\mu$ is called the boundary measure of $u$. 
\end{lemma}
\begin{remark}
The result of Damek and Ricci is valid for all positive eigenfunctions of $\mathcal L$, namely, positive eigenfunctions with eigenvalue $\beta^2-\rho^2$, $\beta\in\R$. However, our results do not apply when $\beta\in (-\infty, 0]$.
\end{remark}
Next, we consider the natural action of the subgroup $A$ on $S$ (see (\ref{nonisotropic})):
\begin{equation}\label{aaction}
r\cdot (n,a)=(\delta_r(n),ra),\:\:\:\:r\in A,\:(n,a)\in S.
\end{equation}
\begin{remark}\label{invarianceofadmissible}
We note that for each $\alpha>0$, the admissible domain $\Gamma_{\alpha}(\underline{0})$ is invariant under this action. Indeed, if $(n,a)\in\Gamma_{\alpha}(\underline{0})$, then for every $r>0$,
\begin{equation*}
d(\delta_r(n))=rd(n)<ra.
\end{equation*} 
\end{remark}
Given a function $F$ on $S$ and $r>0$, we define the dilation $F_r$ of $F$ by
\begin{equation*}
F_r(n,a):=F\left(\delta_r(n),ra\right),\:\:\:\:(n,a)\in S.
\end{equation*}
Given a measure $\nu$ on $N$ and $r>0$, we also define the dilate $\nu_r$ of $\nu$ by
\begin{equation}\label{dilatemna}
\nu_r(E)=r^{-Q}\nu\left(\delta_r(E)\right),
\end{equation}
for every Borel set $E\subseteq N$. 
These notions of dilation are crucial for us primarily because of the following two reasons, which we present in the next two lemmas.
\begin{lemma}\label{dilationoffunc}
Let $\beta>0$. If $F$ is an $\mathcal{L}^{\beta}$-harmonic functions on $S$ then so is $F_r$, for every $r>0$.
\end{lemma}
\begin{proof}
Observe that $F_r$ is the left translation of $F$ by $(\underline{0},r)\in S$. The proof now  follows trivially as $\mathcal{L}^{\beta}$ is left $S$-invariant.
\end{proof}
\begin{lemma}\label{dilateofmeasure}
Let $\beta>0$. If $\nu\in M_{\beta}$, then for each $r>0$, \begin{equation*}
\mathcal{Q}_{i\beta}[\nu_r](n,a)=\mathcal{Q}_{i\beta}[\nu](\delta_r(n),ra),\:\:\:\:\text{for all}\:\:\:(n,a)\in S.
\end{equation*}
\end{lemma}
\begin{proof}
For $E\subseteq N$ a Borel set, it follows from the definition of $\nu_r$ (\ref{dilatemna}) that
\begin{equation*}
\int_{N}\chi_E\:d\nu_r=r^{-Q}\nu\left(\delta_r(E)\right)=r^{-Q}\int_{N}\chi_{\delta_r(E)}(n)\:d\nu(n)=r^{-Q}\int_{N}\chi_{E}\left(\delta_{r^{-1}}(n)\right)\:d\nu(n).
\end{equation*}
Hence, for all nonnegative measurable functions $f$ on $N$ we have  
\begin{equation*}
\int_{N}f(n)\:d\nu_r(n)=r^{-Q}\int_{N}f\left(\delta_{r^{-1}}(n)\right)\:d\nu(n).
\end{equation*}
Thus, from the definition of $\mathcal{Q}_{i\beta}$ (see \ref{qlamconvmu}) we get that for all $(n,a)\in S$,
\begin{eqnarray*}
\mathcal{Q}_{i\beta}[\nu_r](n,a)&=&a^{-Q}\int_{N}q^{i\beta}\left(\delta_{a^{-1}}(n_1^{-1}n)\right)\:d\nu_r(n_1)\\&=&a^{-Q}r^{-Q}\int_{N}q^{i\beta}\left(\delta_{a^{-1}}\left(\delta_{r^{-1}}(n_1^{-1})n\right)\right)\:d\nu(n_1)
\\&=&(ra)^{-Q}\int_{N}q^{i\beta}\left(\delta_{a^{-1}}\left(\delta_{r^{-1}}\left(n_1^{-1}\delta_r(n)\right)\right)\right)\:d\nu(n_1)\\&=&(ra)^{-Q}\int_{N}q^{i\beta}\left(\delta_{(ra)^{-1}}\left(n_1^{-1}\delta_r(n)\right)\right)\:d\nu(n_1)\\&=&\mathcal{Q}_{i\beta}[\nu](\delta_r(n),ra).
\end{eqnarray*}
This completes the proof.
\end{proof}
\section{Main theorem}
We shall first prove a special case of our main result. The proof of the main result will follow by reducing matters to this special case.
\begin{theorem}\label{specialthna}
Suppose that $u$ is a positive eigenfunction of $\mathcal{L}$ on $S$ with eigenvalue $\beta^2-\rho^2$, where $\beta>0$, and that $L\in[0,\infty)$. If the boundary measure $\mu$ of $u$ is finite then the following statements hold.
\begin{enumerate}
\item[(i)]If there exists $\theta>0$, such that
\begin{equation}\label{etadmissible}
\lim_{\substack{a\to 0\\(n,a)\in \Gamma_{\theta}(\underline{0})}}a^{\beta-\rho}u(n,a)=L,
\end{equation}
then $D\mu(\underline{0})=L$.
\item[(ii)]If $D\mu(\underline{0})=L$, then the function $(n,a)\mapsto a^{\beta-\rho}u(n,a)$ has admissible limit $L$ at $\underline{0}$.
\end{enumerate}
\end{theorem}
\begin{proof}
We first prove (i). We choose a $d$-ball $B_0\subset N$, a sequence of positive numbers $\{r_j\}_{j\in\N}$ converging to zero and consider the quotient
\begin{equation*}
L_j=\frac{\mu\left(\delta_{r_j}(B_0)\right)}{m\left(\delta_{r_j}(B_0)\right)},\:\:j\in\N.
\end{equation*}
Assuming (\ref{etadmissible}), we will prove that $\{L_j\}$ is a bounded sequence and every convergent subsequence of $\{L_j\}$ converges to $L$. We first choose a positive number $s$ such that $B_0$ is contained in the $d$-ball $B(\underline{0},s)$. Then, using (\ref{measuredilation}), we get for all $j\in\N$, 
\begin{equation}\label{ljna}
L_j\leq \frac{\mu\left(\delta_{r_j}(B(\underline{0},s))\right)}{m\left(\delta_{r_j}(B_0)\right)}=\frac{\mu\left(\delta_{r_j}(B(\underline{0},s))\right)}{m\left(\delta_{r_j}(B(\underline{0},s))\right)}\times\frac{m(B(\underline{0},s))}{m(B_0)}\leq \frac{m(B(\underline{0},s))}{m(B_0)}M_{HL}(\mu)(\underline{0}).
\end{equation}
We first need to show that $M_{HL}(\mu)(\underline{0})$ is finite. Since $\mu$ is the boundary measure for $u$ we have from (\ref{repofu}) and the relation between $\mathcal {P}_{i\beta}$ and $\mathcal{Q}_{i\beta}$ given in (\ref{plamdaandqlamda}) that
\begin{equation}\label{representation}
a^{\beta-\rho}u(n,a)=Ca^{2\beta}+\mathcal{Q}_{i\beta}[\mu](n,a),\:\:\:\:\text{for all}\:\:\:(n,a)\in S,
\end{equation}
for some positive constant $C$. Since we are interested in the limit as $a$ tends to zero, we may and do assume that $C$ is zero. Therefore, we can rewrite (\ref{etadmissible}) as 
\begin{equation*}
\lim_{\substack{a\to 0\\(n,a)\in \Gamma_{\theta}(\underline{0})}}\mathcal{Q}_{i\beta}[\mu](n,a)=L.
\end{equation*}
This implies, in particular, that 
\begin{equation*}
\lim_{a\to 0}\mathcal{Q}_{i\beta}[\mu](\underline{0},a)=L,
\end{equation*}
and hence there exists a positive number $\delta$ such that
\begin{equation*}
\sup_{0<a<\delta}\mathcal{Q}_{i\beta}[\mu](\underline{0},a)<\infty.
\end{equation*}
Since $\mu$ is a finite measure, using boundedness of the function $q^{i\beta}$ we also have that for all $a\geq\delta$,
\begin{equation*}
\mathcal{Q}_{i\beta}[\mu](\underline{0},a)=a^{-Q}\int_{N}q^{i\beta}\left(\delta_{a^{-1}}(n_1^{-1})\right)\:d\mu(n_1)\leq C'a^{-Q}\int_{N}\:d\mu(n_1)\leq C'\delta^{-Q}\mu(N).
\end{equation*}
Combining the above two inequalities, we obtain
\begin{equation*}
\sup_{a>0}\mathcal{Q}_{i\beta}[\mu](\underline{0},a)<\infty.
\end{equation*}Lemma \ref{maximalna} now implies that $M_{HL}(\mu)(\underline{0})$ is finite. Boundedness of the sequence $\{L_j\}$ is now a consequence of the inequality (\ref{ljna}). We now choose a convergent subsequence of $\{L_j\}$ and denote it also, for the sake of simplicity, by $\{L_j\}$. For $j\in\N$, we define a function $F_j$ on $S$ by
\begin{equation*}
F_j(n,a)=F\left(\delta_{r_j}(n),r_ja\right),\:\:\:\:\:\: (n,a)\in S,
\end{equation*}
where, as in Lemma \ref{lbetaharmonic},
\begin{equation*}
F(n,a)=a^{\beta-\rho}u(n,a)=\mathcal{Q}_{i\beta}[\mu](n,a).
\end{equation*}
Lemma \ref{lbetaharmonic} now implies that $F$ is $\mathcal{L}^{\beta}$-harmonic and hence by Lemma \ref{dilationoffunc}, $F_j$ is $\mathcal{L}^{\beta}$-harmonic for each $j\in \N$. We now claim that $\{F_j\}$ is locally bounded. To prove this claim, we choose a compact set  $K\subset S$. Then there exists a positive number $\alpha$ such that $K$ is contained in $\Gamma_{\alpha}(\underline{0})$. Indeed, we consider the map
\begin{equation*}
(n,a)\mapsto\frac{d(n)}{a},\:\:\:(n,a)\in K.
\end{equation*}
Using continuity of this map and compactness of $K$, we get some $\alpha>0$ such that
\begin{equation*}
\frac{d(n)}{a}<\alpha,\:\:\:\text{for all}\:\:\:(n,a)\in K,
\end{equation*} 
that is, $(n,a)\in \Gamma_{\alpha}(\underline{0})$. Using the invariance of $\Gamma_{\alpha}(\underline{0})$ under the action (\ref{aaction}) (see Remark \ref{invarianceofadmissible}) and Lemma \ref{maximalna}, we obtain that
\begin{equation*}
\sup_j\sup_{(n,a)\in \Gamma_{\alpha}(\underline{0})}F_j(n,a)\leq\sup_{(n,a)\in \Gamma_{\alpha}(\underline{0})}F(n,a)=\sup_{(n,a)\in \Gamma_{\alpha}(\underline{0})}\mathcal{Q}_{i\beta}[\mu](n,a)\leq c_{\alpha}M_{HL}(\mu)(\underline{0}).
\end{equation*}
Hence, in particular, $\{F_j\}$ is a locally bounded sequence of $\mathcal{L}^{\beta}$-harmonic functions on $S$. Applying Lemma \ref{montelna} (generalization of Montel's theorem), we extract a subsequence $\{F_{j_k}\}$ of ${F_j}$ which converges normally to a $\mathcal{L}^{\beta}$-harmonic function $g$ on $S$. We now show that $g$ is identically equal to $L$ in $\Gamma_{\theta}(\underline{0})$. To show this, we take $(n_0,a_0)\in \Gamma_{\theta}(\underline{0})$. Since $\{r_{j_k}\}$ converges to zero as $k$ goes to infinity and $F(n,a)=a^{\beta-\rho}u(n,a)$ has limit $L$, as $(n,a)$ tends to $(\underline{0},0)$ within $\Gamma_{\theta}(\underline{0})$,
\begin{equation*}
g(n_0,a_0)=\lim_{k\to\infty}F_{j_k}(n_0,a_0)=\lim_{k\to\infty}F\left(\delta_{r_{j_k}}(n_0),r_{j_k}a_0\right)=L.
\end{equation*}
So, $g$ is the constant function $L$ on $\Gamma_{\theta}(\underline{0})$. Since $\mathcal{L}^{\beta}$ is elliptic with real analytic coefficient, it follows that $g$ is real analytic. Since $\Gamma_{\theta}(\underline{0})$ is open in $S$, we get 
\begin{equation}\label{Flimit}
g(n,a)=L,\:\:\:\text{for all}\:\:\:(n,a)\in S.
\end{equation}
We now consider the dilate $\mu_{r_{j_k}}$ of $\mu$ according to (\ref{dilatemna}). By Lemma \ref{dilateofmeasure}, we have that
\begin{equation}\label{Fdilate}
F_{j_k}(n,a)=F\left(\delta_{r_{j_k}}(n),r_{j_k}a\right)=\mathcal{Q}_{i\beta}[\mu]\left(\delta_{r_{j_k}}(n),r_{j_k}a\right)=\mathcal{Q}_{i\beta}[\mu_{r_{j_k}}](n,a), 
\end{equation}
for all $(n,a)\in S$. It now follows from (\ref{Flimit}) and (\ref{Fdilate}) that $\mathcal{Q}_{i\beta}[\mu_{r_{j_k}}]$ converges normally to the constant function $L$ which is same as $\mathcal{Q}_{i\beta}[Lm]$. Lemma \ref{normalna} then implies that the sequence of positive measures $\{\mu_{r_{j_k}}\}$ converges to the positive measure $Lm$ in weak*. We then apply Lemma \ref{mthna} to conclude that $\{\mu_{r_{j_k}}(B)\}$ converges to $Lm(B)$ for every  $d$-ball $B\subset N$. Using this for $B=B_0$, we get that
\begin{equation*}
Lm(B_0)=\lim_{k\to \infty}\mu_{r_{j_k}}(B_0)=\lim_{k\to\infty}{r_{j_k}}^{-Q}\mu\left(\delta_{{r_{j_k}}}(B_0)\right)=m(B_0)\lim_{k\to\infty}\frac{\mu\left(\delta_{{r_{j_k}}}(B_0)\right)}{m\left(\delta_{{r_{j_k}}}(B_0)\right)}.
\end{equation*}
This implies that the sequence $\{L_{j_{k}}\}$ converges to $L$ and hence so does $\{L_j\}$, as $\{L_j\}$ is convergent. Thus, every convergent subsequence of  the bounded sequence $\{L_j\}$ converges to $L$. This implies that $\{L_j\}$ itself converges to $L$. Since $B_0$ and $\{r_j\}$ are arbitrary it follows that $\mu$ has strong derivative $L$ at $\underline{0}$. 

We now prove (ii). We suppose that $D\mu(\underline{0})$ is equal to $L$. Since the admissible limit of the function $(n,a)\mapsto a^{2\beta}$ is $0$, without loss of generality we assume as before that $C$ is zero in (\ref{representation}). We need to prove that the admissible limit of the function 
\begin{equation*}
F(n,a):=a^{\beta-\rho}u(n,a)=\mathcal{Q}_{i\beta}[\mu](n,a),\:\:(n,a)\in S,
\end{equation*}
at $\underline{0}$ is equal to $L$. We fix a positive number $\alpha$ and a sequence $\{(n_j,a_j)\mid j\in\N\}\subset \Gamma_{\alpha}(\underline{0})$ such that $\{a_j\}$ converges to zero. Since $D\mu(\underline{0})$ is $L$, it follows, in particular, that
\begin{equation*}
\lim_{r\to 0}\frac{\mu(B(\underline{0},r))}{m(B(\underline{0},r))}=L.
\end{equation*}
Therefore, there exists some positive constant $\delta$ such that 
\begin{equation*}
\sup_{0<r<\delta}\frac{\mu(B(\underline{0},r))}{m(B(\underline{0},r))}<L+1.
\end{equation*}
Finiteness of the measure $\mu$ implies that for all $r\geq\delta$,
\begin{equation*}
\frac{\mu(B(\underline{0},r))}{m(B(\underline{0},r))}\leq\frac{\mu(N)}{m(B(\underline{0},1))\delta^Q}.
\end{equation*}
The above two inequalities together with Lemma \ref{maximalna} implies that
\begin{equation*}
\sup_{(n,a)\in \Gamma_{\alpha}(\underline{0})}F(n,a)=\sup_{(n,a)\in \Gamma_{\alpha}(\underline{0})}\mathcal{Q}_{i\beta}[\mu](n,a)\leq C_{\alpha,\beta}M_{HL}(\mu)(\underline{0})<\infty.
\end{equation*}
In particular, $\{F(n_j,a_j)\}$ is a bounded sequnce. We consider a convergent subsequence of this sequence, denote it also, for the sake of simplicity, by $\{F(n_j,a_j)\}$ such that 
\begin{equation}\label{Flimitprime}
\lim_{j\to\infty}F(n_j,a_j)=L^{\prime}.
\end{equation}
It suffices to prove that $L'$ is equal to $L$. Using the sequence $\{a_j\}$, we define for each $j\in\N$, 
\begin{equation*}
F_j(n,a)=F\left(\delta_{a_j}(n),a_ja\right),\:\:\:(n,a)\in S.
\end{equation*}
As we have shown in the first part, we can prove that $\{F_j\}$ is a locally bounded sequence of $\mathcal{L}^{\beta}$-harmonic functions on $S$. Hence, by Lemma \ref{montelna}, there exists a subsequence $\{F_{j_k}\}$ of $\{F_j\}$ which converges normally to a positive $\mathcal{L}^{\beta}$-harmonic function $g$ on $S$. By defining 
\begin{equation*}
v(n,a):=a^{\rho-\beta}g(n,a),\:\:\:(n,a)\in S,
\end{equation*}
we get from Lemma \ref{dilationoffunc} that $v$ is a positive eigenfunction of $\mathcal{L}$ with eigenvalue $\beta^2-\rho^2$. Hence, by Lemma \ref{positiveeigen} and (\ref{plamdaandqlamda}), there exists a unique positive measure $\nu$ on $N$ and a unique nonnegative constant $C'$ such that
\begin{equation*}
v(n,a)=C'a^{\beta+\rho}+a^{\rho-\beta}\mathcal{Q}_{i\beta}[\nu](n,a),\:\:\:\:\text{for all}\:\:(n,a)\in S.
\end{equation*}This implies that \begin{equation}\label{repofg}
g(n,a)=C'a^{2\beta}+\mathcal{Q}_{i\beta}[\nu](n,a),\:\:\:\:\text{for all}\:\:(n,a)\in S.
\end{equation}
Applying Lemma \ref{maximalna} once again, we observe that
\begin{equation*}
\sup_{a>0}\sup_jF_j(\underline{0},a)\leq\sup_{a>0}F(\underline{0},a)=\mathcal{Q}_{i\beta}[\mu](\underline{0},a)\leq c_{\alpha}M_{HL}(\mu)(\underline{0})<\infty.
\end{equation*}
This shows that \begin{equation*}
\sup_{a>0}g(\underline{0},a)<\infty,
\end{equation*}an hence we must have $C'=0$ in (\ref{repofg}). Considering the dilate $\mu_{a_{j_k}}$ of $\mu$ according to (\ref{dilatemna}), we see by using Lemma \ref{dilateofmeasure} that for all $(n,a)\in S$
\begin{equation*}
F_{j_k}(n,a)=F\left(\delta_{a_{j_k}}(n),r_{j_k}a\right)=\mathcal{Q}_{i\beta}[\mu]\left(\delta_{a_{j_k}}(n),r_{j_k}a\right)=\mathcal{Q}_{i\beta}[\mu_{a_{j_k}}](n,a).
\end{equation*}Therefore, in view of (\ref{repofg}), we conclude that $\mathcal{Q}_{i\beta}[\mu_{a_{j_k}}]$ converges to $\mathcal{Q}_{i\beta}[\nu]$, normally on $S$. By Lemma \ref{normalna}, we thus obtain weak* convergence of $\{\mu_{j_k}\}$ to $\nu$. Since $D\mu(\underline{0})=L$, it follows that for any $d$-ball $B\subset N$,
\begin{equation*}
\lim_{k\to\infty}\mu_{j_k}(B)=\lim_{k\to\infty}{a_{j_k}}^{-Q}\mu(\delta_{{a_{j_k}}}(B))=\lim_{k\to\infty}\frac{\mu(\delta_{{a_{j_k}}}(B))}{m(\delta_{{a_{j_k}}}(B)}m(B)=Lm(B).
\end{equation*}
Hence by Lemma \ref{mthna}, $\nu=Lm$. As $g=\mathcal{Q}_{i\beta}[\nu]$, it follows that
\begin{equation*}
g(n,a)=L,\:\:\:\text{for all}\:\:\:(n,a)\in S.
\end{equation*}
This, in turn, implies that $\{F_{j_k}\}$ converges to the constant function $L$ normally on $S$. On the other hand, we note that
\begin{equation*}
F(n_{j_k},a_{j_k})=F\left(\delta_{{a_{j_k}}}\left(\delta_{{a^{-1}_{j_k}}}(n_{j_k})\right),a_{j_k}\right)=F_{j_{k}}\left(\delta_{{a^{-1}_{j_k}}}(n_{j_k}),1\right).
\end{equation*}
As $(n_{j_k},a_{j_k})$ belongs to the admissible region $\Gamma_{\alpha}(\underline{0})$, for all $k\in\N$, it follows that
\begin{equation*}
\left(\delta_{{a^{-1}_{j_k}}}(n_{j_k}),1\right)\in\overline B(\underline{0},\alpha)\times\{1\},
\end{equation*}
which is a compact subset of $S$. Therefore,
\begin{equation*}
\lim_{k\to\infty}F(n_{j_k},a_{j_k})=\lim_{k\to\infty}F_{j_{k}}\left(\delta_{{a^{-1}_{j_k}}}(n_{j_k}),1\right)=L,
\end{equation*}
as the convergence is uniform on $\overline B(\underline{0},\alpha)\times\{1\}$. In view of (\ref{Flimitprime}), we can thus conclude that $L'=L$. This completes the proof.
\end{proof}
We now state and prove our main result.
\begin{theorem}\label{mainthna}
Suppose that $u$ is a positive eigenfunction of $\mathcal{L}$ in $S$ with eigenvalue $\beta^2-\rho^2$, where $\beta>0$, and that $n_0\in N$, $L\in[0,\infty)$. If $\mu$ is the boundary measure of $u$ then the following statements hold.
\begin{enumerate}
\item[(i)]If there exists $\theta>0$, such that
\begin{equation*}
\lim_{\substack{a\to 0\\(n,a)\in \Gamma_{\theta}(n_0)}}a^{\beta-\rho}u(n,a)=L,
\end{equation*}
then $D\mu(n_0)=L$.
\item[(ii)]If $D\mu(n_0)=L$, then the function $(n,a)\mapsto a^{\beta-\rho}u(n,a)$ has admissible limit $L$ at $n_0$.
\end{enumerate}
\end{theorem}
\begin{proof}
We consider the translated measure $\mu_0=\tau_{n_0}\mu$, where
\begin{equation*}
\tau_{n_0}\mu (E)=\mu(En_0^{-1} ),\:\:\:\text{for all Borel subsets}\:\:\:E\subseteq N.
\end{equation*}
Using translation invariance of $m$, it follows from the definition of strong derivative (Definition \ref{impdefnna}, ii)) that $D\mu_0(\underline{0})$ and $D\mu(n_0)$ are equal. As in the previous theorem, we may and do suppose that $C=0$ in (\ref{representation}). Thus, we can rewrite the representation formula (\ref{representation}) for $u$ as
\begin{equation*}
a^{\beta-\rho}u(n,a)=\mathcal{Q}_{i\beta}[\mu](n,a),\:\:\:(n,a)\in S.
\end{equation*}\label{transna}
Since $\mathcal{Q}_{i\beta}[\mu_0](\cdot,a)=\mu_0\ast q^{i\beta}_a$, $a\in A$, it follows that
\begin{equation*}
\mathcal{Q}_{i\beta}[\mu_0](n,a)=(\tau_{n_0}\mu)\ast q^{i\beta}_a(n)=\mu\ast q^{i\beta}_a(n_0n)=\mathcal{Q}_{i\beta}[\mu](n_0n,a),
\end{equation*}
for all $(n,a)\in S$. We fix an arbitrary positive number $\alpha$. Using the left-invariance of the quasi-metric $\bf{d}$, we have $(n,a)\in \Gamma_{\alpha}(\underline{0})$ if and only if $(n_0n,a)\in \Gamma_{\alpha}(n_0)$. Thus, we conclude from the last equation that
\begin{equation*}
\lim_{\substack{a\to 0\\(n,a)\in \Gamma_{\alpha}(\underline{0})}}\mathcal{Q}_{i\beta}[\mu_0](n,a)=\lim_{\substack{a\to 0\\(n,a)\in \Gamma_{\alpha}(n_0)}}\mathcal{Q}_{i\beta}[\mu](n,a).
\end{equation*}
Hence, it suffices to prove the theorem under the assumption that $n_0=\underline{0}$. We now show that we can even take $\mu$ to be finite. Let $\tilde{\mu}$ be the restriction of $\mu$ on $\overline{B(\underline{0},\tau^{-1})}$. Suppose $B(n,s)$ is any given $d$-ball in $N$. Then for all $r\in (0,\left[\tau^2(s+d(n))\right]^{-1})$, it follows that $\delta_r(B(n,s))$ is a subset of $B(\underline{0},\tau^{-1})$. Indeed, if $n_1\in\delta_r(B(n,s))=B(\delta_r(n),rs)$, then we have
\begin{equation*}
d(\underline{0},n_1)\leq \tau\left[d(\underline{0},\delta_r(n))+d(\delta_r(n),n_1)\right]\leq \tau\left[rd(n)+rs\right]<\tau^{-1}.
\end{equation*}
This implies that $D\mu(\underline{0})$ and $D\tilde{\mu}(\underline{0})$ are equal. We now claim that
\begin{equation}\label{finaleqna}
\lim_{\substack{a\to 0\\(n,a)\in \Gamma_{\alpha}(\underline{0})}}\mathcal{Q}_{i\beta}[\mu](n,a)=\lim_{\substack{a\to 0\\(n,a)\in \Gamma_{\alpha}(\underline{0})}}\mathcal{Q}_{i\beta}[\tilde{\mu}](n,a).
\end{equation}
In order to prove this claim, we first observe that
\begin{equation}\label{integrallimitzero}
\lim_{a\to 0}\int_{B(\underline{0},\tau^{-1})^c}q^{i\beta}_a(n_1^{-1}n)\:d\mu(n_1)=0,
\end{equation}
uniformly for $n\in B(\underline{0},1/(2\tau^2))$. To prove this observation, we first note that for $n\in B(\underline{0},1/(2\tau^2))$ and $n_1\in B(\underline{0},\tau^{-1})^c$,
\begin{equation*}
d(n)<(2\tau^2)^{-1}=(2\tau)^{-1}\tau^{-1}\leq(2\tau)^{-1}d(n_1).
\end{equation*}
Thus, using (\ref{reversetriangle}) and the inequality above, we obtain for $n\in B(\underline{0},1/(2\tau^2))$ and $n_1\in  B(\underline{0},\tau^{-1})^c$,
\begin{equation}\label{triangleesti}
d(n_1^{-1}n)=d(n^{-1}n_1)\geq\frac{1}{\tau}d(n_1)-d(n)\geq\frac{1}{\tau}d(n_1)-\frac{d(n_1)}{2\tau}
=\frac{d(n_1)}{2\tau}.\end{equation}
Recalling the expression (\ref{qbetaexpression}) of $q_a^{i\beta} $, we get for $n=(X,Z)\in B(\underline{0},1/(2\tau^2))$ that
\begin{eqnarray*}
	&&\int_{B(\underline{0},\tau^{-1})^c}q^{i\beta}_a(n_1^{-1}n)\:d\mu(n_1)\\
	&=&\int_{B(\underline{0},\tau^{-1})^c}\frac{c_{\beta}\:a^{2\beta}}{\left(16a^2+8a\|X-X_1\|^2+d\left((X_1,Z_1)^{-1}(X,Z)\right)\right)^{\rho+\beta}}\:d\mu(X_1,Z_1)\\
	&\leq& c_{\beta}\:a^{2\beta}\int_{B(\underline{0},\tau^{-1})^c}\frac{1}{\left(16a^2+d(n_1^{-1}n)^2\right)^{\beta+\rho}}\:d\mu(n_1)\\
	&\leq&c_{\beta}\:a^{2\beta}\int_{B(\underline{0},\tau^{-1})^c}\frac{1}{\left(16a^2+\frac{d(n_1)^2}{4\tau^2}\right)^{\beta+\rho}}\:d\mu(n_1)\:\:\:\:\:\:(\text{using the inequality (\ref{triangleesti})})
	\\&=&c_{\beta}\:a^{2\beta}\int_{B(\underline{0},\tau^{-1})^c}\left(\frac{16+\frac{d(n_1)^2}{4\tau^2}}{16a^2+\frac{d(n_1)^2}{4\tau^2}}\right)^{\beta+\rho}\frac{1}{\left(16+\frac{d(n_1)^2}{4\tau^2}\right)^{\beta+\rho}}\:d\mu(n_1)\\
	&\leq&c_{\beta}\:a^{2\beta}\int_{B(\underline{0},\tau^{-1})^c}\left(\frac{64\tau^2}{d(n_1)^2}+1\right)^{\beta+\rho}\frac{1}{\left(16+\frac{d(n_1)^2}{4\tau^2}\right)^{\beta+\rho}}\:d\mu(n_1)\\&\leq&c_{\beta,\tau}\:a^{2\beta}\int_{B(\underline{0},\tau^{-1})^c}\frac{1}{\left(16+\frac{d(n_1)^2}{4\tau^2}\right)^{\beta+\rho}}\:d\mu(n_1).
\end{eqnarray*} 
Applying Lemma \ref{integralfinite} for $a=1$, the integral in the last inequality is finite. Hence, as $\beta>0$, letting $a$ goes to zero on the right-hand side of the last inequality, (\ref{integrallimitzero}) follows. Now,
\begin{equation*}
\mathcal{Q}_{i\beta}[\mu](n,a)=\mathcal{Q}_{i\beta}[\tilde{\mu}](n,a)+\int_{B(\underline{0},\tau^{-1})^c}q^{i\beta}_a(n_1^{-1}n)\:d\mu(n_1).
\end{equation*}
We take $\epsilon>0$. By (\ref{integrallimitzero}), we get some positive number $a_1$ such that for all $a\in (0,a_1)$, the integral on the right-hand side of the equality above is smaller than $\epsilon$, for all $n\in B(\underline{0},1/(2\tau^2))$. On the other hand, we note that
\begin{equation*}
\Gamma_{\alpha}(\underline{0})\cap\{(n,a)\in S\mid a<1/(2\alpha \tau^2)\}\subset B\left(\underline{0},1/(2\tau^2)\right)\times\{(n,a)\in S\mid a<1/(2\alpha \tau^2)\}.
\end{equation*}
Hence, for all $(n,a)\in \Gamma_{\alpha}(\underline{0})$ with $a<\min\{a_1,1/(2\alpha \tau^2)\}$, we have \begin{equation*}
\mathcal{Q}_{i\beta}[\mu](n,a)-\mathcal{Q}_{i\beta}[\tilde{\mu}](n,a)<\epsilon.
\end{equation*}This proves (\ref{finaleqna}). Therefore, as $\alpha>0$ is arbitrary, we may and do suppose that $\mu$ is a finite measure. The proof now follows from Theorem \ref{specialthna}.
\end{proof}
As an immediate consequence of Theorem \ref{mainthna} we have the following.
\begin{corollary}
Suppose that $u$ is a positive eigenfunction of $\mathcal{L}$ on $S$ with eigenvalue $\beta^2-\rho^2$, where $\beta>0$, and that $n_0\in N$, $L\in[0,\infty)$. If for some $\theta>0$,
\begin{equation*}
\lim_{\substack{a\to 0\\(n,a)\in \Gamma_{\theta}(n_0)}}a^{\beta-\rho}u(n,a)=L,
\end{equation*}
then for every $\alpha>0$,\begin{equation*}
\lim_{\substack{a\to 0\\(n,a)\in \Gamma_{\alpha}(n_0)}}a^{\beta-\rho}u(n,a)=L.
\end{equation*}
\end{corollary}
\begin{remark}
Since $N$ has been assumed to be noncommutative, the class of harmonic $NA$ groups don't contain the real hyperbolic spaces. However, the obvious analogue of Theorem \ref{mainthna} for real hyperbiloc spaces $\mathbb H^l=\{(x,y)\mid x\in \R^{l-1}, y\in (0,\infty)\}$, $l\geq 2$, also holds true. Taking into account that in this case $Q=2\rho=l-1$ and
\begin{eqnarray*}
\mathcal L=y^2\left(\Delta_{\R^{l-1}}+\frac{\partial^2}{\partial y^2} \right)-(l-2)y\frac{\partial}{\partial y},\\
\mathcal L^{\beta}=y^2\left(\Delta_{\R^{l-1}}+\frac{\partial^2}{\partial y^2} \right)-(2\beta-1)y\frac{\partial}{\partial y},\:\:\:\: \beta>0,
\end{eqnarray*}  
with 
\begin{equation*}
P(x)=c_l (1+\|x\|^2)^{-(l-1)},\:\:x\in\R^{l-1},
\end{equation*}
the proof can be carried out exactly as before.
\end{remark}

\section*{acknowledgements}
The authors would like to thank Pratyoosh Kumar for several useful discussions during the preparation of this work. The second named author is supported by a research fellowship from Indian Statistical Institute.

\end{document}